%% file: SPCGS.tex
\documentclass{article}

\usepackage[nonatbib,final]{neurips_2020}
\usepackage[utf8]{inputenc} % allow utf-8 input
\usepackage[T1]{fontenc}    % use 8-bit T1 fonts
\usepackage{hyperref}       % hyperlinks
\usepackage{url}            % simple URL typesetting
\usepackage{booktabs}       % professional-quality tables
\usepackage{amsfonts}       % blackboard math symbols
\usepackage{nicefrac}       % compact symbols for 1/2, etc.
\usepackage{microtype}      % microtypography
\usepackage{graphicx}
\usepackage{enumitem}
\input{math.tex}

\usepackage[numbers,sort]{natbib}
\title{Efficient Projection-Free Algorithms for Saddle Point Problems}

\author{
	Cheng Chen$^1$ \qquad Luo Luo$^2$\thanks{Corresponding author} \qquad Weinan Zhang$^1$ \qquad Yong Yu$^1$ \\ 
	$^1$Shanghai Jiao Tong University \\
	$^2$The Hong Kong University of Science and Technology  \\
	{\small \texttt{jack\_chen1990@sjtu.edu.cn}\hskip1.8em
		\texttt{luoluo@ust.hk}\hskip1.8em
		\texttt{\{wnzhang,yyu\}@apex.sjtu.edu.cn}}
}

\begin{document}
	
	\maketitle
	
	\begin{abstract}
		The Frank-Wolfe algorithm is a classic method for constrained optimization problems. It has recently been popular in many machine learning applications because its projection-free property leads to more efficient iterations. In this paper, we study projection-free algorithms for   convex-strongly-concave saddle point problems with complicated constraints. Our method combines Conditional Gradient Sliding  with Mirror-Prox and shows that it only requires $\tilde{\cO}(1/\sqrt{\epsilon})$ gradient evaluations and $\tilde{\cO}(1/\epsilon^2)$ linear optimizations in the batch setting. We also extend our method to the stochastic setting and propose first stochastic projection-free algorithms for saddle point problems. Experimental results demonstrate the effectiveness of our algorithms and verify our theoretical guarantees.
	\end{abstract}
	
	\section{Introduction}
	In this paper, we study the following saddle point problems:
	\begin{equation*} \label{eq:problem}
	\min_{\x\in\cX}\max_{\y\in\cY} f(\x,\y)    
	\end{equation*}
	where the objective function $f(\x,\y)$ is convex-concave and $L$-smooth;  $\cX$ and $\cY$ are convex and compact sets. Besides this general form, we also consider the stochastic minimax problem:
	\begin{align}\label{eq:stochastic}
	\min_{\x\in\cX}\max_{\y\in\cY} f(\x,\y)\triangleq\BE_\xi[F(\x,\y;\xi)],
	\end{align}
	where $\xi\in\Xi$ is a random variable. One popular specific setting of (\ref{eq:stochastic}) is the finite-sum case where $\xi$ is sampled from a finite set $\Xi=\{\xi_i\}_{i=1}^n$. Denoting $F_i(\x,\y)\triangleq F(\x,\y;\xi_i)$, we can write the objective function as
	\begin{align}\label{eq:finitesum}
	f(\x,\y)\triangleq\frac{1}{n}\sum_{i=1}^n F_i(\x,\y).
	\end{align}
	%where each $\f_i(\x,\y)$ is differentiable and $n$, the number of training examples, is usually very large.
	
	We are interested in the cases where the feasible set is complicated such that projecting onto $\cX\times\cY$ is rather expensive or even intractable. One example of such case is the nuclear norm ball constraint, which is widely used in machine learning applications such as multiclass classification~\cite{dudik2012lifted}, matrix completion~\cite{candes2009exact,jaggi2010simple,lacoste2013affine}, 
	factorization machine~\cite{lin2018online},
	polynomial neural nets~\cite{livni2014computational} 
	and two-player games whose strategy space contains a large number of constraints~\cite{ahmadinejad2019duels}.
	%Another example of such case is two-player games. When the strategy space contains a large number of constraints~\cite{ahmadinejad2019duels}, it is intractable to project onto such feasible sets. %We refer reader to ~\cite{gidel2017frank} for more examples.
	
	The Frank-Wolfe (FW) algorithm~\cite{frank1956algorithm} (a.k.a. conditional gradient method) is initially proposed for constrained convex optimization. It has recently become popular in the machine learning community because of its projection-free property~\cite{jaggi2013revisiting}. The Frank-Wolfe algorithm calls a linear optimization (LO) oracle at each iteration, which is usually much faster than projection for complicated feasible sets. Recently, FW-style algorithms for convex and nonconvex minimization problems has been widely studied~\cite{lacoste2013affine,lan2016conditional,hazan2016variance,hazan2012projection,qu2018non,reddi2016stochastic,yurtsever2019conditional,shen2019complexities,xie2020efficient,zhang2020one,hassani2019stochastic}. However, the only known projection-free algorithms for minimax optimization are for very special cases (e.g. the saddle point belongs to the interior of the feasible set~\cite{gidel2017frank}).
	%there is only a little literature focuses on projection-free algorithms for saddle point problems. 
	
	In this paper, we propose a projection-free algorithm, which we refer to as Mirror-Prox Conditional Gradient Sliding (MPCGS), for convex-strongly-concave saddle point problems. 
	Our method leverages the idea from some projection-type methods~\cite{nemirovski2004prox,thekumparampil2019efficient}, which is based on proximal point iterations. By combining the idea of Mirror-Prox~\cite{thekumparampil2019efficient} with the conditional gradient sliding (CGS)~\cite{lan2016conditional}, MPCGS only requires at most $\tilde{\cO}(1/\sqrt{\epsilon})$ exact gradient evaluations and $\tilde{\cO}(1/\epsilon^2)$ linear optimizations to guarantee $\mathcal O(\epsilon)$ suboptimality error in expectation. 
	We also extend our framework to the stochastic setting and propose Mirror-Prox Stochastic Conditional Gradient Sliding (MPSCGS), which requires to compute at most $\tilde{\cO}(1/\epsilon^2)$ stochastic gradients and call the LO oracle for at most $\tilde{\cO}(1/\epsilon^2)$ times. To the best of our knowledge, MPSCGS is the first stochastic projection-free algorithm for convex-strongly-concave saddle point problems. 
	%To validate our theoretical results, 
	We also conduct experiments on several real-world data sets for robust optimization problem to validate our theoretical analysis. The empirical results show that the proposed methods outperform previous projection-free and projection-based methods when the feasible set is complicated.
	
	\paragraph{Related Works}
	Most existing works on constrained minimax optimization solve the problem with projection. We only provide some representative literature. For the batch setting, the classical extragradient method~\cite{korpelevich1976extragradient} considered a more general variational inequality (VI). \citet{nemirovski2004prox} proposed Mirror-Prox method which achieves a convergence rate of $\cO(1/\epsilon)$ for solving VI. Recently, \citet{thekumparampil2019efficient} improved the convergence rate to $\cO(1/\sqrt{\epsilon})$ when the objection function is strongly-convex-concave. For the stochastic setting, \citet{chavdarova2019reducing,palaniappan2016stochastic} adopted the variance reduction methods to obtain linear convergence rate for strongly-convex-strongly-concave objection functions.
	
	The projection-free methods for saddle point problems are very few. ~\citet{hammond1984solving} found that FW algorithm with a step size $\cO(1/k)$ converges for VI when the feasible set is strongly convex. Recently, ~\citet{gidel2017frank} proposed SP-FW algorithm for strongly-convex-strongly-concave saddle point problem, which achieves a linear convergence rate under the condition that the saddle point belongs to the interior of the feasible set and the condition number is small enough. 
	They also provided an away-step Frank-Wolfe variant~\cite{lacoste2013affine}, called SP-AFW, to address the polytope constraints. However, SP-AFW has to store history information and perform extra operations in each iteration. ~\citet{roy2019online} extended SP-FW to the zeroth-order setting and studied a gradient-free projection-free algorithm which has the theoretical guarantee under the same assumptions on the objective function as SP-FW. \citet{he2015semi} proposed a projection-free algorithm for non-smooth composite saddle point problem. Their method requires to call a composite LO oracle, which is not suitable for general case.
	
	Some recent works focus on hybrid algorithms which combine projection-based and projection-free methods. 
	For example, \citet{juditsky2016solving,cox2017decomposition} transformed a VI with complicated constraints to a ``dual'' VI which is projection-friendly. \citet{lan2013complexity,nouiehed2019solving} solved the saddle point problem by running projection-free methods on $\cX$ and performing projection on $\cY$. In contrast, our methods are purely projection-free.
	
	\paragraph{Paper Organization}
	In Section \ref{sec:pre}, we provide preliminaries and relevant backgrounds. We present our results for the batch setting and stochastic setting in Section \ref{sec:mpcgs} and Section \ref{sec:mpscgs} respectively.
	We give empirical results for our algorithm in Section \ref{sec:exp}, followed by a conclusion in Section \ref{sec:con}.

	\section{Preliminaries and Backgrounds} \label{sec:pre}
	In this section, we first present some notation and assumptions used in this paper. Then we introduce oracle models which are necessary to our methods, followed by an example of application. After that we provide some properties of saddle point problems.% and give the definition of $\epsilon$-saddle point solution. 
	Finally, we introduce CGS and its variants, which are used in our algorithms.

	\subsection{Notation and Assumptions}
	
	Given a differentiable function $f(\x,\y)$, we use $\nabla_\x f(\x,\y)$ (or $\nabla_\y f(\x,\y)$) to denote the partial gradient of $f$ with respect to $\x$ (or $\y$) and define $f_\cS(\x,\y)=\frac{1}{|\cS|}\sum_{\xi\in\cS} F(\x,\y;\xi)$. We use the notation $\tilde{\cO}$ to hide logarithmic factors in the complexity and denote $[n]=\{1,2,\dots,n\}$.
	
	%\subsection{Assumption}
	We impose the following assumptions for our method.
	\begin{assumption} \label{ass1}
		We assume the saddle point problem (\ref{eq:problem}) satisfies: 
		%(a) $f(\x,\y)$ is $L$-smooth; (b) $f(\cdot,\y)$ is convex for every $\y\in\cY$ and  $f(\x,\cdot)$ is $\mu$-strongly concave for every $\x\in\cX$; (c) $\cX$ and $\cY$ are convex compact sets with diameter $D_\cX$ and $D_\cY$ respectively.
		\begin{itemize}[leftmargin=0.6cm]
			\item $f(\x,\y)$ is $L$-smooth, i.e., for every $(\x_1,\y_1), (\x_2,\y_2)\in\cX\times\cY$, it holds that
			\begin{align*}
			\|\nabla f(\x_1,\y_1)-\nabla f(\x_2,\y_2)\|^2\leq L^2\left(\|\x_1-\x_2\|^2+\|\y_1-\y_2\|^2\right).
			\end{align*}
			\item $f(\cdot,\y)$ is convex for every $\y\in\cY$, i.e., for any $\x_1$, $\x_2$ and $\y$, it holds that
			\begin{align*}
			f(\x_1,\y)-f(\x_2,\y)\geq \nabla_\x f(\x_2,\y)^\top (\x_1-\x_2).
			\end{align*}
			\item  $f(\x,\cdot)$ is $\mu$-strongly concave for every $\x\in\cX$, i.e., for any $\x$, $\y_1$ and $\y_2$, it holds that 
			\begin{align*}
			f(\x,\y_1)-f(\x,\y_2)\leq \nabla_\y f(\x,\y_2)^\top (\y_1-\y_2) - \frac{\mu}{2}\|\y_1-\y_2\|^2.
			\end{align*}
			\item $\cX$ and $\cY$ are convex compact sets with diameter $D_\cX$ and $D_\cY$ respectively.
		\end{itemize}
	\end{assumption}
	We use $\kappa=L/\mu$ to denote the condition number.
	
	\begin{assumption} \label{ass2}
		In the stochastic setting, we make the following additional assumptions:
		\begin{itemize}[leftmargin=0.6cm]
			\item $\BE[\nabla F(\x,\y;\xi)]=\nabla f(\x,\y)$ for every $(\x,\y)\in\cX\times\cY$ and $\xi\in\Xi$.
			\item $\BE\|\nabla F(\x,\y;\xi)-\nabla f(\x,\y)\|^2\leq \sigma^2$ for every  $(\x,\y)\in\cX\times\cY$, $\xi\in\Xi$ and constant $\sigma>0$.
			\item $f(\x,\y)$ is $L$-average smooth, i.e., for every $(\x_1,\y_1), (\x_2,\y_2)\in\cX\times\cY$ and $\xi\in\Xi$, it holds that
			\begin{align*}
			\BE\|\nabla F(\x_1,\y_1,\xi)-\nabla F(\x_2,\y_2,\xi)\|^2\leq L^2(\|\x_1-\x_2\|^2+\|\y_1-\y_2\|^2).
			\end{align*}
		\end{itemize}
		%The first assumption indicates that $\nabla F(\x,\y,\xi)$ is an unbiased estimator of the gradient. The second assumption means that the variance is bounded. The last one assumes $f$ has an averaged $L$-Lipschitz gradient.
	\end{assumption}
	
	%All assumptions are mild and frequently used in the analysis of projection-free algorithms and stochastic optimization.
	
	In the convex-concave setting, for any $(\hat{\x},\hat{\y})\in\cX\times\cY$, we have the following inequality:
	\begin{equation*}
	\min_{\x\in\cX} f(\x,\hat{\y})\leq f(\hat{\x},\hat{\y}) \leq \max_{\y\in\cY} f(\hat{\x},\y).
	\end{equation*}
	Furthermore, Problem (\ref{eq:problem}) has at least one saddle point solution $(\x^*,\y^*)\in\cX\times\cY$ which satisfies:
	\begin{equation*}
	\min_{\x\in\cX} f(\x,\y^*)= f(\x^*,\y^*) = \max_{\y\in\cY} f(\x^*,\y).
	\end{equation*}
	We measure the suboptimality error by the primal-dual gap: $\max_{\y\in\cY} f(\hat{\x},\y)-\min_{\x\in\cX} f(\x,\hat{\y})$, which is widely used in saddle point problems. We further define $\epsilon$-saddle point as follows:
	\begin{definition}
		A point $(\hat{\x},\hat{\y})\in\cX\times\cY$ is an $\epsilon$-saddle point of a convex-concave function $f$ if:
		\begin{align}\label{dfn:saddle}
		max_{\y\in\cY}f(\hat{\x},\y)-\min_{\x\in\cX} f(\x,\hat{\y})\leq \epsilon. 
		\end{align}
	\end{definition}
	Notice that \citeauthor{gidel2017frank} \cite{gidel2017frank} adopted a different criterion: $w(\hat{\x},\hat{\y})=f(\hat{\x},\y^*)-f(\x^*,\hat{\y})$. It is obvious that the left-hand side of (\ref{dfn:saddle}) is an upper bound of $w(\hat{\x},\hat{\y})$.

	\subsection{Oracle models}
	In this paper, we consider the following oracles for different settings:
	\begin{itemize}[leftmargin=0.6cm]
		\item First Order Oracle (FO): Given $(\x,\y)\in\cX\times\cY$, the FO returns $f(\x,\y)$ and $\nabla f(\x,\y)$.
		\item Stochastic First Order Oracle (SFO): For a function $\BE_\xi[F(\x,\y;\xi)]$ where $\xi\sim P$, SFO returns $F(\x,\y;\xi')$ and $\nabla F(\x,\y;\xi')$ where $\xi'$ is a sample drawn from $P$.
		\item Incremental First Order Oracle (IFO): In the finite-sum setting, IFO takes a sample $i\in[n]$ and returns $F_i(\x,\y)$ and $\nabla F_i(\x,\y)$.
		\item Linear Optimization Oracle (LO): Given a vector $\g\in\BR^d$ and a convex and compact set $\Omega\subseteq\BR^d$, the LO returns a solution of the problem
		$\min_{\v\in\Omega} \langle\v, \g\rangle$.
	\end{itemize}
	
	\subsection{Example Application: Robust Optimization for Multiclass Classification} \label{sec:app}
	
	We consider the multiclass classification problem with $h$ classes. Suppose the training set is $\cD=\{(\a_i,b_i)\}_{i=1}^n$, where $\a_i\in\BR^d$ is the feature vector of the $i$-th sample and $b_i\in[h]$ is the corresponding label. The goal is to find an accurate linear predictor with parameter $\X = [\x_1^\top, \x_2^\top, \cdots, \x_h^\top]\in\BR^{h\times d}$ that predicts $b=\argmax_{j\in[h]} \x_j^\top\a$ for any input feature vector $\a\in\BR^d$. 
	
	The robust optimization model~\cite{namkoong2017variance} with multivariate logistic loss~\cite{dudik2012lifted,zhang2012accelerated} under nuclear norm ball constraint can be formulated as the following convex-concave minimax optimization:
	\begin{align} \label{eq:multi}
	\min_{\X\in\cX}\max_{\y\in\cY} f(\X,\y) \triangleq \frac{1}{n}\sum_{i=1}^n y_i\log\left(1+\sum_{j\neq y_i}\left(\x_j^\top\a_i-\x_{y_i}^\top\a_i\right)\right)-\frac{\lambda}{2}\|n\y-{\bf 1}_n\|_2^2,
	\end{align}
	where 
	$\cX=\left\{\X\in\BR^{h\times d}: \|\X\|_* \leq \tau \right\}$ and $\cY=\left\{\y\in\BR^n : y_i\geq 0, \sum_{i=1}^n y_i = 1 \right\}$.
	It is obvious that the objective function (\ref{eq:multi}) is convex-strongly-concave, which satisfies our assumptions. In this case, projecting onto $\cX$ requires to perform full SVD, which takes $\cO(hd\min\{h,d\})$ time. On the other hand, the linear optimization on $\cX$ only needs to find the top singular vector, whose cost is linear to the number of non-zero entries in the gradient matrix.
	
	\subsection{Conditional Gradient Sliding}
	Conditional Gradient Sliding (CGS)~\cite{lan2016conditional} is a projection-free algorithm for convex minimization. It leverages Nesterov's accelerate gradient descent~\cite{nesterov27method} to speed-up Frank-Wolfe algorithms. For strongly-convex objective function, CGS only requires $\cO(\sqrt{\kappa}\log(1/\epsilon))$ FO calls and $\cO(1/\epsilon)$ LO calls to find an $\epsilon$-suboptimal solution. We present the details of CGS in Algorithm \ref{alg:cgs}. Notice that the $k$-th iteration of CGS considers the following sub-problem 
	\begin{align*}
	\min_{\u\in\Omega}~\langle\nabla h(\w_k),\u \rangle+\frac{\beta_k}{2}\|\u-\u_{k-1}\|^2,
	\end{align*}
	which can be efficiently solved by the conditional gradient method in Algorithm \ref{alg:cndg}. \citeauthor{lan2016conditional} \cite{lan2016conditional} also extended CGS to stochastic setting and proposed stochastic conditional gradient sliding (SCGS). Later, \citet{hazan2016variance} proposed STOchastic
	variance-Reduced Conditional gradient sliding (STORC) for finite-sum setting whose complexities of IFO and LO are  $\cO((n+\kappa^2)\log(1/\epsilon))$ and $\cO(1/\epsilon)$ respectively.
	
	%In this paper, we will leverage CGS to solve a strongly convex subproblem for the batch setting. We propose an inexact version of STORC
	
	\begin{algorithm}[t] 
		\caption{CGS Method for strongly convex functions}\label{alg:cgs}
		{\textbf{Input:}} $L$-smooth and $\mu$-strongly-convex function $h$, convex and compact set $\Omega$, total iterations $N$.
		{\textbf{Input:}} The initial point $\bar{\x}_0\in\Omega$ satisfies $h(\bar{\x}_0)-h(\x^*)\leq \delta_0$.
		\begin{algorithmic}[1]
			%\STATE Draw one sample $\xi$ and compute $\w_0=\argmin_{\w\in\Omega}\nabla \langle H(\v,\xi),\w\rangle$ for some arbitary $\v\in\Omega$
			\STATE $M\leftarrow\sqrt{24L/\mu}$.
			\FOR{$t=1,2,\dots,N$}
			\STATE  $\x_0\leftarrow\bar{\x}_{t-1}$, $\u_0\leftarrow\x_0$
			\FOR{$k=1,2,\dots,M$}
			\STATE $\lambda_k\leftarrow\frac{2}{k+1}$, $\beta_k\leftarrow\frac{2L}{k}$, $\eta_{t,k}\leftarrow\frac{8L\delta_0 2^{-t}}{\mu Nk}$
			\STATE  $\w_k\leftarrow(1-\lambda_k)\x_{k-1}+\lambda_k\u_{k-1}$
			\STATE  $\u_k\leftarrow\text{CndG}(\nabla h(\w_k),\u_{k-1},\beta_k,\eta_{t,k},\Omega)$
			\STATE  $\x_k\leftarrow(1-\lambda_k)\x_{k-1}+\lambda_k\u_k$
			\ENDFOR
			\STATE $\bar{\x}_t=\x_M$
			\ENDFOR
		\end{algorithmic}
	\end{algorithm}
	
	\begin{algorithm}[t] 
		\caption{Procedure $\q^+=$CndG($\r,\q,\beta,\eta,\Omega$)}\label{alg:cndg}
		\begin{algorithmic}[1]
			\STATE  $\q_1\leftarrow\q$.
			\FOR{$t=1,2,\dots$}
			\STATE $\p_t\leftarrow\argmax_{\x\in\Omega}\langle\r+\beta(\q_t-\q),\q_t-\x\rangle$, $\tau_t\leftarrow\langle\r+\beta(\q_t-\q),\q_t-\p_t\rangle$
			\STATE If $\tau_t\leq\eta$, set $\q^+=\q_t$ and  terminate the procedure.
			\STATE $\theta_t\leftarrow\min\left\{1,\frac{\tau_t}{\beta\|\q_t-\p_t\|^2}\right\}$, $\q_{t+1}\leftarrow(1-\theta_t)\q_t+\theta_t\p_t$
			\ENDFOR
		\end{algorithmic}
	\end{algorithm}

	\section{Mirror-Prox Conditional Gradient Sliding} \label{sec:mpcgs}
	For the batch setting of (\ref{eq:problem}), we propose Mirror-Prox Conditional Gradient Sliding (MPCGS), which is presented in Algorithm \ref{alg:mpcgs}. Our MPCGS method combines ideas of Mirror-Prox algorithm~\cite{thekumparampil2019efficient} and CGS method~\cite{lan2016conditional}. The key idea of MPCGS is to solve a proximal problem in each iteration, which makes  $\x_k$ and $\y_k$ satisfy following conditions:
	\begin{itemize}[leftmargin=0.6cm]
		\item $\y_k$ is an $\epsilon_k$-approximate maximizer of $f(\x_k,\cdot)$, i.e., $f(\x_k,\y_k)\geq \max_\y f(\x_k,\y)-\epsilon_k$;
		\item The update of $\x_k$ corresponds to an CGS updating step (Algorithm \ref{alg:cgs}) for $-f(\cdot,\y_k)$, i.e.,
		\begin{align*}
		\v_k= \text{CndG}(\nabla_\x f(\z_k,\y_k),\v_{k-1},\alpha_k,\zeta_k,\cX), \quad \x_k=(1-\gamma_k)\x_{k-1}+\gamma_k\v_k.
		\end{align*}
	\end{itemize}
	The procedure of solving the proximal problem is presented in Algorithm \ref{alg:prox}. In the Prox-step procedure, we iteratively compute an $\epsilon_{cgs}$-approximate maximizer of $f(\x_{r-1},\cdot)$ and then update $\v_r$ and $\x_r$ according to $\y_r$. Since $f(\x,\cdot)$ is smooth and strongly concave for all $\x\in\cX$,  the number of calls to the FO and LO oracles performed by CGS method for finding an $\epsilon_k$-approximate maximizer can be bounded by $\cO(\sqrt{\kappa}\log(1/\epsilon_{cgs}))$  and $\cO(1/\epsilon_{cgs})$ respectively. 
	
	On the other hand, in Algorithm \ref{alg:prox} the CndG procedure computes $\v_r$ as an $\zeta$-approximate solution of the following problem:
	\begin{align*}
	\min_{\u\in\cX}\left\{\nabla_\x f(\z,\y^*(\x_{r-1}))^\top\u+\frac{\alpha}{2}\|\u-\v\|^2\right\}.
	\end{align*}
	Thus, the idealized updating of $\x_r$ in Algorithm \ref{alg:prox} is 
	\begin{align*}
	\x_r=(1-\gamma)\x+\gamma\cdot\argmin_{\u\in\cX}\left\{\nabla_\x f(\z,\y^*(\x_{r-1}))^\top\u+\frac{\alpha}{2}\|\u-\v\|^2\right\}.
	\end{align*}
	Since $\psi(\x)\triangleq(1-\gamma)\x+\gamma\cdot\argmin_{\u\in\cX}\{\nabla_\x f(\z,\y^*(\x_{r-1}))^\top\u+\frac{\alpha}{2}\|\u-\v\|^2\}$ is a $(1/2)$-contraction mapping with a unique fixed point (see the proof of Lemma 2 in the Appendix), the Prox-step procedure only requires $\cO(\log(1/\epsilon))$ iterations if $\epsilon_{cgs}$ and $\zeta$ are small enough. 
	
	The following theorem shows the convergence rate of solving problem (\ref{eq:problem}) by Algorithm \ref{alg:mpcgs}.
	\begin{theorem} \label{thm:csc}
		Suppose the objective function $f(\x,\y)$ satisfies Assumption \ref{ass1}. By setting
		\begin{align*}
		\gamma_k=\frac{3}{k+2},\quad \alpha_k = \frac{6\kappa L}{k+1}, \quad \zeta_k = \frac{LD_\cX^2}{384k(k+1)}, \quad \epsilon_k=\frac{\kappa LD_\cX^2}{k(k+1)(k+2)}
		\end{align*}
		for Algorithm \ref{alg:mpcgs}, then we have
		$$\max_{\y\in\cY}f({\x}_k,\y)-\min_{\x\in\cX} f(\x,\bar{\y}_k)\leq \frac{11\kappa LD_\cX^2}{(k+1)(k+2)}. $$
	\end{theorem}
	Theorem \ref{thm:csc} implies the upper bound complexities of the algorithm as follows.
	\begin{corollary} \label{col:csc}
		Under the same assumption of Theorem \ref{thm:csc},  Algorithm \ref{alg:mpcgs} requires $\tilde{\cO}\left(\kappa\sqrt{LD_\cX^2/\epsilon}\right)$ FO complexity and $\tilde{\cO}\left(\kappa^2L^2D_\cX^2D_\cY^2/\epsilon^2\right)$ LO complexity to achieve an $\epsilon$-saddle point.
	\end{corollary}

	\begin{algorithm}[t] 
		\caption{Mirror-Prox Conditional Gradient Sliding}\label{alg:mpcgs}
		{\textbf{Input:}} Objective function $f(\x,\y)$, parameters $\gamma_k$, $\alpha_k$, $N$, $\x_0$, $\y_0$, $\epsilon_k$ and $\zeta_k$.\\
		{\textbf{Output:}} $\x_N$, $\bar{\y}_N$.
		\begin{algorithmic}[1]
			\STATE $\v_0 \leftarrow \x_0$
			\FOR{$k=1,2,\dots,N$}
			\STATE  $\z_k\leftarrow(1-\gamma_k)\x_{k-1}+\gamma_k\v_{k-1}$.
			\STATE\label{line:mpcgs-sub} $(\x_k,\y_k,\v_k)\leftarrow\text{Prox-step}(f,\x_{k-1},\y_{k-1},\z_k,\v_{k-1},\gamma_k,\alpha_k,\zeta_k,\epsilon_k)$.% ensuring:
			%$$ f(\x_k,\y_k)\geq \max_{\y\in\cY}f(\x_k,\y)-\epsilon_k, $$
			%$$ \v_k= \text{CndG}(\nabla_\x f(\z_k,\y_k),\v_{k-1},\alpha_k,\zeta_k,\cX), \quad \x_k=(1-\gamma_k)\x_{k-1}+\gamma_k\v_k.$$
			%\STATE Let $\y_k=$ CGS($-f(\z_k,\cdot),\y_{k-1},N_k$).
			%\STATE Let $\v_k=$ CndG($\d_k,\v_{k-1},\alpha_k,\zeta_{k},\cX$).
			%\STATE Compute $\x_k=(1-\gamma_k)\x_{k-1}+\gamma_k\v_k$.
			\STATE $\bar{\y}_k\leftarrow\frac{3}{k(k+1)(k+2)}\sum_{s=1}^k s(s+1)\y_s$.
			\ENDFOR
		\end{algorithmic}
	\end{algorithm}
	
	\begin{algorithm}[t]
		\caption{Procedure   ($\x_R,\y_{R},\v_{R}$)$=$Prox-step($f,\x_0,\y_0,\z,\v,\gamma,\alpha,\zeta,\epsilon$)} \label{alg:prox}
		\begin{algorithmic}[1]
			\STATE $\epsilon_{cgs}\leftarrow \epsilon/(64\kappa)$, $\epsilon_{mp}\leftarrow4\gamma\sqrt{2\kappa L\epsilon_{cgs}/\alpha^2+2\zeta/\alpha}$, $R\leftarrow\left\lceil\log_2(4D_\cX/\epsilon_{mp})\right\rceil$.
			\FOR{$r=1,\dots,R$}
			%\STATE $\y_r\leftarrow$CGS($-f(\x_{r-1},\cdot),\y_{k-1},N_k$)
			\STATE Use CGS method (Algorithm \ref{alg:cgs}) with objective function $-f(\x_{r-1},\cdot)$ and start point $\y_0$ to compute $\y_r$ such that:
			$f(\x_{r-1},\y_r)\geq \max_{\y\in\cY}f(\x_{r-1},\y)-\epsilon_{cgs} $
			\STATE $\v_r\leftarrow\text{CndG(}\nabla_\x f(\z,\y_r),\v,\alpha,\zeta,\cX)$,\quad $\x_r\leftarrow(1-\gamma)\x_0+\gamma\v_r$
			\ENDFOR
		\end{algorithmic}
	\end{algorithm}

	\section{Mirror-Prox Stochastic Conditional Gradient Sliding} \label{sec:mpscgs}
	In this section, we extend MPCGS to the stochastic setting (\ref{eq:stochastic}). Recall that we adopt CGS to find $\epsilon$-approximate maximizer of problem $ \max_{\y\in\cY} f(\x_k,\y)$ in the batch setting, which only require logarithmic iterations. In the stochastic case, we would like to use the STORC~\cite{hazan2016variance} algorithm instead. Since the original STORC can only be applied to the finite-sum situation, we have to first study an inexact variant of STORC which does not depend on the exact gradient. Then we leverage the inexact STORC algorithm to establish our projection-free algorithm for stochastic saddle point problems.
	
	\subsection{Inexact Stochastic Variance Reduced Conditional Gradient Sliding}
	We propose Inexact STORC (iSTORC) algorithm to solve the following stochastic convex optimization problem:
	\begin{equation}
	\min_{\x\in\Omega}h(\x)=\BE_\xi[H(\x;\xi)], \label{eq:min-h}   
	\end{equation}
	where $\xi\in\Xi$ is a random variable; the feasible set $\Omega$ is convex, compact and has diameter $D$. We assume that $h(\x)$ is $L$-smooth and $\mu$-strongly convex. We also suppose that the algorithm can access the stochastic gradient $H(\x;\xi)$ which satisfies:
	\begin{itemize}[leftmargin=0.6cm]
		\item $\BE[\nabla H(\x;\xi)]=\nabla h(\x)$, \quad $\forall \x\in\Omega$, $\forall \xi\in\Xi$.
		\item $\BE[\|\nabla H(\x;\xi)-\nabla h(\x)\|^2]\leq \sigma^2, \quad \forall \x\in\Omega, \forall \xi\in\Xi.$
		\item $\BE[\|\nabla H(\x_1;\xi)-\nabla H(\x_2;\xi)\|^2]\leq L^2\|\x_1-\x_2\|^2. \quad \forall (\x_1,\x_2)\in\Omega^2, \forall \xi\in\Xi$
	\end{itemize}
	
	The idea of iSTORC is to approximate the exact gradient in STORC by appropriate number of stochastic gradient samples. The following theorem shows the convergence rate of iSTORC.
	
	\begin{theorem} \label{thm:istorc}
		Running Inexact STORC (Algorithm \ref{alg:istorc}) with the following parameters:
		{\small\begin{align*}
			\lambda_k=\frac{2}{k+1},  \beta_k=\frac{3L}{k},   M=\left\lceil4\sqrt{2\kappa}\right\rceil,  \eta_{t,k}=\frac{\kappa LD^2}{2^{t-2}Mk},  S=4800M\kappa,  Q_t=\left\lceil\frac{1200\cdot 2^{t-1}\sigma^2\sqrt{\kappa}}{L^2D^2}\right\rceil,
			\end{align*}}
		%In finite-sum setting, we choose $S=1200M\kappa$ and $\cQ_t=[n]$. In expectation setting, we choose $S=2400M\kappa$ and stochastic parameter $Q_t=\left\lceil\frac{1200\cdot 2^{t-1}\sigma^2\sqrt{\kappa}}{L^2D^2}\right\rceil$.
		we have 
		\begin{align*}
		\BE[h(\bar{\x}_t)-h(\x^*)]\leq\frac{LD^2}{2^{t+1}}, 
		\end{align*}
		where $\x^*\in\argmin_{\x\in\Omega} h(\x)$.
	\end{theorem}
	
	Theorem \ref{thm:istorc} implies the following upper bound complexities of iSTORC.
	\begin{corollary} \label{col:istorc}
		To achieve $\bar\x_t$ such that $\BE[h(\bar{\x}_t)-h(\x^*)]\leq\epsilon$, iSTORC (Algorithm \ref{alg:istorc}) requires $\cO\left((\sqrt{\kappa}/(L\epsilon )+\kappa^2)\log(LD^2/\epsilon)\right)$ SFO complexity and $\cO\left(LD^2/\epsilon\right)$ LO complexity.
	\end{corollary}
	\begin{remark}
		If the objective function has the finite-sum form, we can choose $\cQ_t=\{\xi_1,\dots,\xi_n\}$ and obtain the same upper complexities bound as STORC. %Compared with STORC, our method does not require each function $h_i(\x)$ is convex and $L$-smooth. We only require an averaged $L$-Lipschitz gradient condition.
	\end{remark}
	\begin{remark}
		Notice that the SFO complexity of SCGS is $\cO(\kappa/(L\epsilon))$. When  $2^{-\sqrt{\kappa}}<\epsilon<\kappa^{-1.5}$,
		iSTORC has better SFO complexity than SCGS.
	\end{remark}

	\begin{algorithm}[t] 
		\caption{Inexact STORC (iSTORC)} \label{alg:istorc}
		{\textbf{Input:}} $L$-smooth and $\mu$-strongly convex function $h(\x)$, initial point $\bar{\x}_0\in\Omega$ and total iteration $N$.\\
		{\textbf{Input:}} Parameters $\gamma_k$, $\beta_k$, $Q_t$, $S$, $M$ and $\eta_{t,k}$.
		
		\begin{algorithmic}[1]
			%\STATE Draw one sample $\xi$ and compute $\w_0=\argmin_{\w\in\Omega}\nabla \langle H(\v,\xi),\w\rangle$ for some arbitary $\v\in\Omega$
			\FOR{$t=1,2,\dots,N$}
			\STATE  $\x_0\leftarrow\bar{\x}_{t-1}$, $\u_0\leftarrow\x_0$.
			\STATE Draw $Q_t$ samples $\cQ_t=\{\xi_{j}\}_{j=1}^{Q_t}$, and compute $\vnu\leftarrow\nabla h_{\cQ_t}(\x_0)$.
			\FOR{$k=1,2,\dots,M$}
			\STATE  $\w_k\leftarrow(1-\lambda_k)\x_{k-1}+\lambda_k\u_{k-1}$.
			\STATE Draw $S$ samples $\cS_{t,k}=\{\xi_{t,k,j}\}_{j=1}^{S}$ and  compute $\r_k\leftarrow\nabla h_{\cS_{t,k}}(\w_k)-\nabla h_{\cS_{t,k}}(\x_0)+\vnu$.
			\STATE $\u_k\leftarrow$CndG($\r_k,\u_{k-1},\beta_k,\eta_{t,k},\Omega$), $\x_k\leftarrow(1-\lambda_k)\x_{k-1}+\lambda_k\u_k$.
			\ENDFOR
			\STATE $\bar{\x}_t\leftarrow\x_M$.
			\ENDFOR
		\end{algorithmic}
	\end{algorithm}

	\subsection{Mirror-Prox Stochastic Conditional Gradient Sliding}
	We present  our Mirror-Prox Stochastic Conditional Gradient Sliding (MPSCGS) in Algorithm \ref{alg:mpscgs}. The idea of MPSCGS is similar to that of MPCGS. The main difference is that we solve the proximal problem in MPSCGS by a stochastic proximal step, where we adopt the proposed iSTORC algorithm. Specifically, in each iteration we ensures that $\x_k$ and $\y_k$ satisfy following conditions:
	\begin{itemize}[leftmargin=0.6cm]
		\item $\y_k$ is an $\epsilon_k$-approximate maximizer of $f(\x_k,\cdot)$ in expectation, i.e.,
		\[\BE[f(\x_k,\y_k)]\geq \BE[\max_{\y\in\cY}f(\x_k,\y)]-\epsilon_k;\]
		\item The update of $\x_k$ and $\v_k$ ensures that
		\begin{align*}
		\v_k= \text{CndG}(\nabla_\x f_{\cP_k}(\z_k,\y_k),\v_{k-1},\alpha_k,\zeta_k,\cX), \quad \x_k=(1-\gamma_k)\x_{k-1}+\gamma_k\v_k.
		\end{align*}
	\end{itemize}

	The following theorem shows the convergence rate of solving problem (\ref{eq:stochastic}).
	\begin{theorem} \label{thm:mpscgs}
		Suppose the objective function $f(\x,\y)$ satisfies Assumption \ref{ass1} and \ref{ass2}. If we set
		$$\gamma_k=\frac{3}{k+2},\ \  \alpha_k = \frac{6\kappa L}{k+1}, \ \  \zeta_k = \frac{LD_\cX^2}{576k(k+1)}, \ \  \epsilon_k=\frac{\kappa LD_\cX^2}{k(k+1)(k+2)},\ \  P_k=\left\lceil\frac{96\sigma^2(k+1)^3}{\kappa L^2D_\cX^2} \right\rceil $$
		for Algorithm \ref{alg:mpscgs}, then we have
		$$\BE\left[\max_{\y\in\cY}f({\x}_k,\y)-\min_{\x\in\cX} f(\x,\bar{\y}_k)\right]\leq \frac{12\kappa LD_\cX^2}{(k+1)(k+2)}. $$
	\end{theorem}
	Theorem \ref{thm:mpscgs} implies the following corollary of oracle complexity.
	\begin{corollary} \label{col:mpscgs-f}
		Under the assumption in Theorem \ref{thm:mpscgs} and the assumption that objective function has finite-sum form of (\ref{eq:finitesum}),  Algorithm \ref{alg:mpscgs} needs $\tilde{\cO}\big((n+\kappa^2)\sqrt{\kappa LD_\cX^2/\epsilon}+\kappa\sigma^2 D_\cX^2/\epsilon^2\big)$ IFO complexity and  $\tilde{\cO}\left(\kappa^2L^2D_\cX^2D_\cY^2/\epsilon^2\right)$ LO complexity to achieve an $\epsilon$-saddle point.
	\end{corollary}
	\begin{corollary} \label{col:mpscgs-e}
		Under the assumption in Theorem \ref{thm:mpscgs} and the assumption that objective function has the expectation form of (\ref{eq:stochastic}),  Algorithm \ref{alg:mpscgs} needs $\tilde{\cO}\left(\kappa^{2.5}\sqrt{ LD_\cX^2/\epsilon}{+}(\kappa^{1.5}{+}\kappa\sigma^2)D_\cX^2/\epsilon^2\right)$  SFO complexity and  $\tilde{\cO}\left(\kappa^2L^2D_\cX^2D_\cY^2/\epsilon^2\right)$ LO complexity to achieve an $\epsilon$-saddle point.
	\end{corollary}

	\begin{algorithm}[t] 
		\caption{Mirror-Prox Stochastic Conditional Gradient Sliding}\label{alg:mpscgs}
		{\textbf{Input:}} Objective function $f(\x,\y)$, parameters $\gamma_k$, $\alpha_k$, $N$, $\x_0$, $\y_0$, $\epsilon_k$, $P_k$ and $\zeta_k$.\\
		{\textbf{Output:}} $\x_N$, $\bar{\y}_N$.
		\begin{algorithmic}[1]
			\STATE $\v_0 \leftarrow \x_0$.
			\FOR{$k=1,2,\dots,N$}
			\STATE  $\z_k\leftarrow(1-\gamma_k)\x_{k-1}+\gamma_t\v_{k-1}$.
			\STATE Draw $P_k$ samples $\cP_k=\{\xi_{j}\}_{j=1}^{P_k}$.
			\STATE $(\x_k,\y_k,\v_k)\leftarrow\text{Stochastic-Prox-step}(f,\x_{k-1},\y_{k-1},\z_k,\v_{k-1},\gamma_k,\alpha_k,\zeta_k,\cP_k,\epsilon_k)$.%, ensuring:
			%$$ \BE[f(\x_k,\y_k)]\geq \BE[\max_{\y\in\cY}f(\x_k,\y)]-\epsilon_k, $$
			%$$ \v_k= \text{CndG}(\nabla_\x f_{\cP_k}(\z_k,\y_k),\v_{k-1},\alpha_k,\zeta_k,\cX), \quad \x_k=(1-\gamma_k)\x_{k-1}+\gamma_k\v_k.$$
			%\STATE Let $\y_k=$ CGS($-f(\z_k,\cdot),\y_{k-1},N_k$).
			%\STATE Let $\v_k=$ CndG($\d_k,\v_{k-1},\alpha_k,\zeta_{k},\cX$).
			%\STATE Compute $\x_k=(1-\gamma_k)\x_{k-1}+\gamma_k\v_k$.
			\STATE $\bar{\y}_k\leftarrow\frac{3}{k(k+1)(k+2)}\sum_{s=1}^k s(s+1)\y_s$.
			\ENDFOR
		\end{algorithmic}
	\end{algorithm}
	
	\begin{algorithm}[t]
		\caption{Procedure   ($\x_R,\y_{R},\v_{R}$)$=$Stochastic-Prox-step($f,\x_0,\y_0,\z,\v,\gamma,\alpha,\zeta,\cP,\epsilon$)} \label{alg:sprox}
		\begin{algorithmic}[1]
			\STATE $\epsilon_{cgs}\leftarrow\frac{\epsilon}{64\kappa}$, $\epsilon_{mp}\leftarrow8\gamma^2(\frac{4\kappa L\epsilon_{cgs}}{\alpha^2}+\frac{2\zeta}{\alpha}+\frac{2\sigma^2}{|\cP|\alpha^2})$, $R\leftarrow\left\lceil\log_2\frac{4D_\cX^2}{\epsilon_{mp}}\right\rceil$.
			\FOR{$r=1,\dots,R$}
			%\STATE $\y_r\leftarrow$CGS($-f(\x_{r-1},\cdot),\y_{k-1},N_k$)
			\STATE Use iSTORC method (Algorithm \ref{alg:istorc}) with objective function $-f(\x_{r-1},\cdot)$ and start point $\y_0$ to compute $\y_r$ such that:
			$ \BE[f(\x_{r-1},\y_r)]\geq \BE[\max_{\y\in\cY}f(\x_{r-1},\y)]-\epsilon_{cgs} $
			\STATE $\v_r\leftarrow\text{CndG(}\nabla_\x f_\cP(\z,\y_r),\v,\alpha,\zeta,\cX)$,\quad $\x_r\leftarrow(1-\gamma)\x_0+\gamma\v_r$
			\ENDFOR
		\end{algorithmic}
	\end{algorithm}
	
	% \begin{algorithm} \label{alg:sp-cgs}
	% 	\caption{SP-CGS for convex-strongly-concave saddle point problems.}
	% 	{\textbf{Input:}} Objective function $f(\x,\y)$.\\
	% 	{\textbf{Input:}} Parameters $\gamma_k$, $\alpha_k$, $C_k$ and $\zeta_k$.
	
	% 	\begin{algorithmic}[1]
	% 		\FOR{$k=1,2,\dots,$}
	% 		\STATE Compute $\z_k=(1-\gamma_k)\x_{k-1}+\gamma_t\v_{k-1}$.
	% 		\STATE Let $\y_k=$ CGS($-f(\z_k,\cdot),\y_{k-1},N_k$).
	% 		\STATE Let $\v_k=$ CndG($\d_k,\v_{k-1},\alpha_k,\zeta_{k},\cX$).
	% 		\STATE Compute $\x_k=(1-\gamma_k)\x_{k-1}+\gamma_k\v_k$.
	% 		\ENDFOR
	% 	\end{algorithmic}
	% \end{algorithm}
	
	% \begin{algorithm} \label{alg:sp-cgs-sc}
	% 	\caption{SP-CGS for strongly-convex-strongly-concave saddle point problems.}
	% 	{\textbf{Input:}} Objective function $f(\x,\y)$.\\
	% 	{\textbf{Input:}} Parameters $\gamma_k$, $\alpha_k$, $C_k$ and $\zeta_k$.
	
	% 	\begin{algorithmic}[1]
	% 		\FOR{$t=1,2,\dots,$}
	% 		\STATE Let $\x_0=\bar{\x}_{t-1}$, $\y_0=\bar{\y}_{t-1}$, $\v_0=\x_0$.
	% 		\FOR{$k=1,2,\dots,K_t$}
	% 		\STATE Compute $\z_k=(1-\gamma_k)\x_{k-1}+\gamma_t\v_{k-1}$.
	% 		\STATE Let $\y_t=$ CGS($-f(\z_t,\cdot),\y_{t-1},N_{t,k}$).
	% 		\STATE Let $\v_k=$ CndG($\d_k,\v_{k-1},\alpha_k,\zeta_{t,k},\cX$).
	% 		\STATE Compute $\x_k=(1-\gamma_k)\x_{k-1}+\gamma_k\v_k$.
	% 		\ENDFOR
	% 		\STATE $\bar{\x}_t=\x_{K_t}$, $\bar{\y}_t=$ CGS($-f(\bar{\x}_t,\cdot),\y_{K_t},N_t$).
	% 		\ENDFOR
	% 	\end{algorithmic}
	% \end{algorithm}
	
	\section{Experiments} \label{sec:exp}
	In this section, we empirically evaluate the performance of our methods on the robust multiclass classification problem introduced in Section \ref{sec:app}. Specifically, we choose the nuclear norm ball with radius $\tau=100$ and the regularization parameter $\lambda=1/n$. We compare our methods with saddle point Frank-Wolfe (SPFW)~\cite{gidel2017frank} and stochastic variance reduce extragradient (SVRE)~\cite{chavdarova2019reducing}. SPFW is a projection-free algorithm as discussed before, while SVRE is the state-of-the-art projection-based stochastic methods for saddle point problems. We conduct experiments on three real-world data sets from the LIBSVM repository\footnote{https://www.csie.ntu.edu.tw/~cjlin/libsvmtools/datasets/}: rcv1 ($n=15,564$, $d=47,236$, $h=53$), sector ($n=6,412$, $d=55,197$, $h=105$) and news20 ($n=15,935$, $d=62,061$, $h=20$).
	
	Since the primal-dual gap is hard to compute, we evaluate algorithms by the following FW-gap~\cite{jaggi2013revisiting}:
	$$\cG(\x,\y)=\max_{\u\in\cX}\langle\x-\u, \nabla_\x f(\x,\y)\rangle+ \max_{\v\in\cY}\langle\y-\v, -\nabla_\y f(\x,\y)\rangle.$$
	which is an upper bound of primal-dual gap and easy to compute. We measure the actual running time rather than number of iterations because the computational cost of projection, linear optimization and computing gradients are quite different.
	
	We implement the mini-batch version of SVRE with batch size $100$. The learning rate of SVRE is searched in $\{10^{-1},10^{-2},\dots,10^{-6}\}$. On the other hand, the parameters of projection-free methods follows what the theory suggests. We report the experimental result in Figure \ref{fig:exp}. 
	
	In all experiments, our methods outperform baselines. The SVRE only performs a few iterations due to its heavy computational cost of the projection on to the trace norm ball. SPFW converges slowly for it does not have theoretical guarantee on the convex-strongly-concave case. We also find that MPSCGS converges faster than MPCGS, because the stochastic algorithms take advantages when $n$ is very large.
	
	\begin{figure}[ht]
		\centering 
		\begin{tabular}{ccc}
			\includegraphics[scale=0.33]{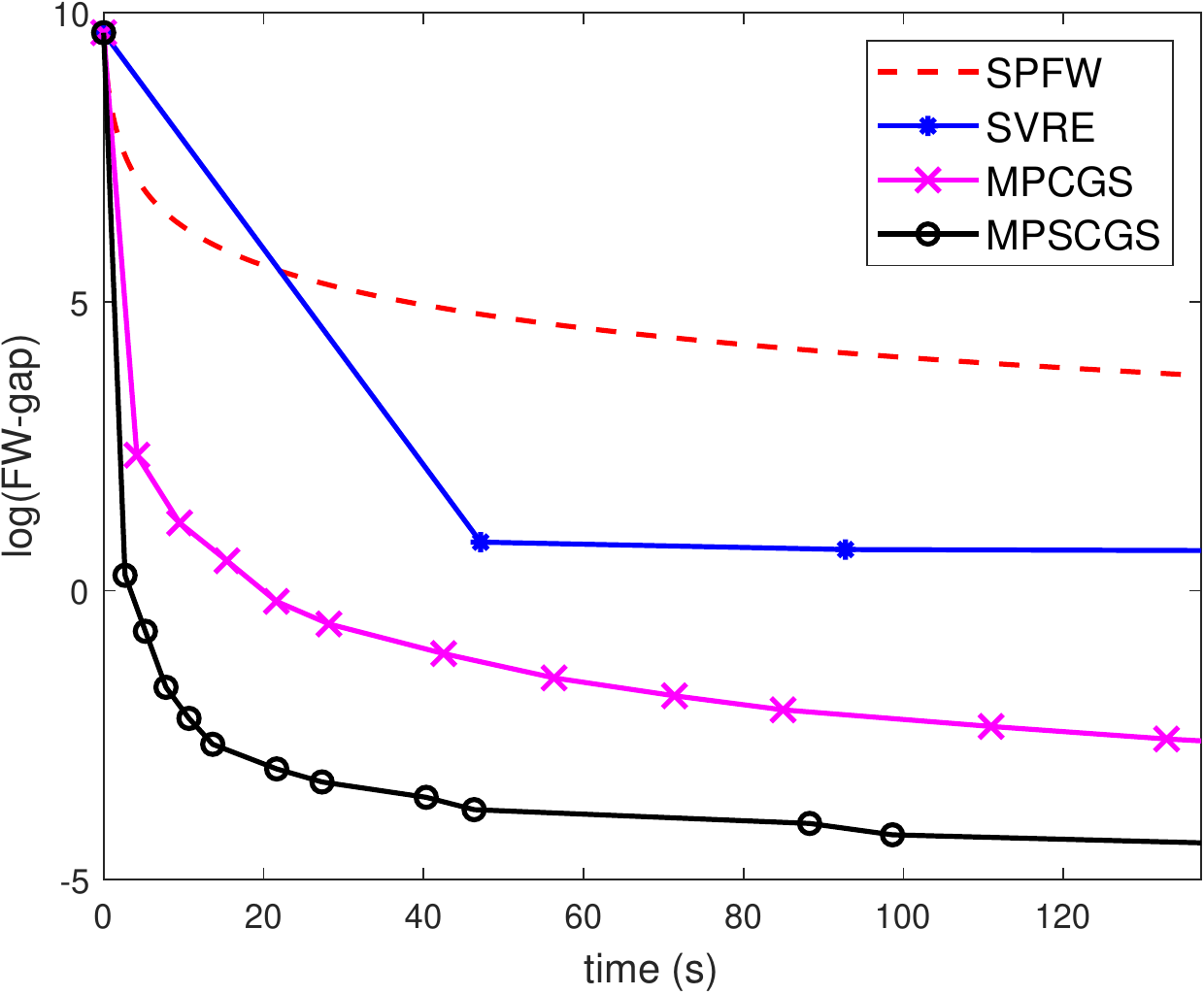} &
			\includegraphics[scale=0.33]{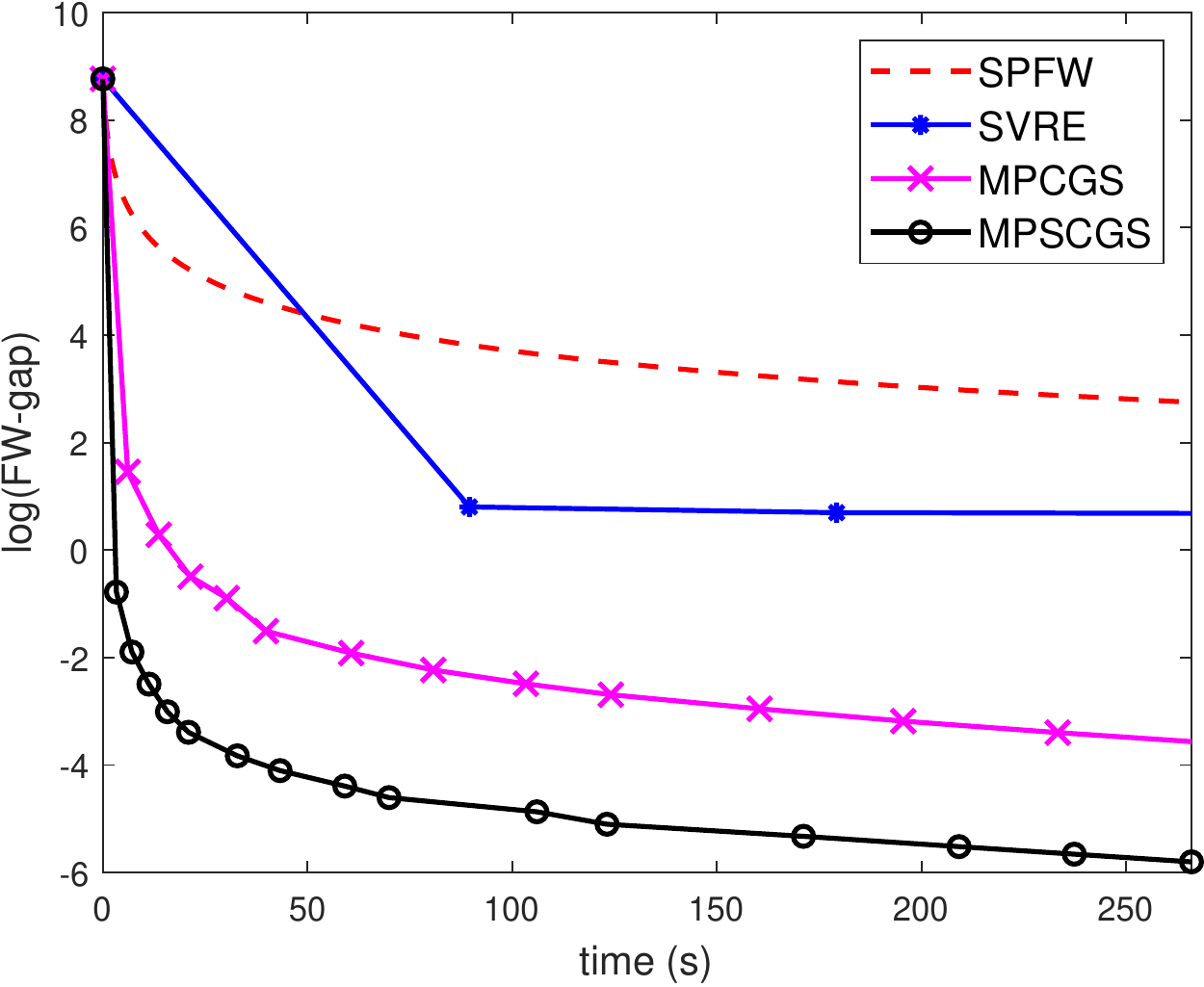} &
			\includegraphics[scale=0.33]{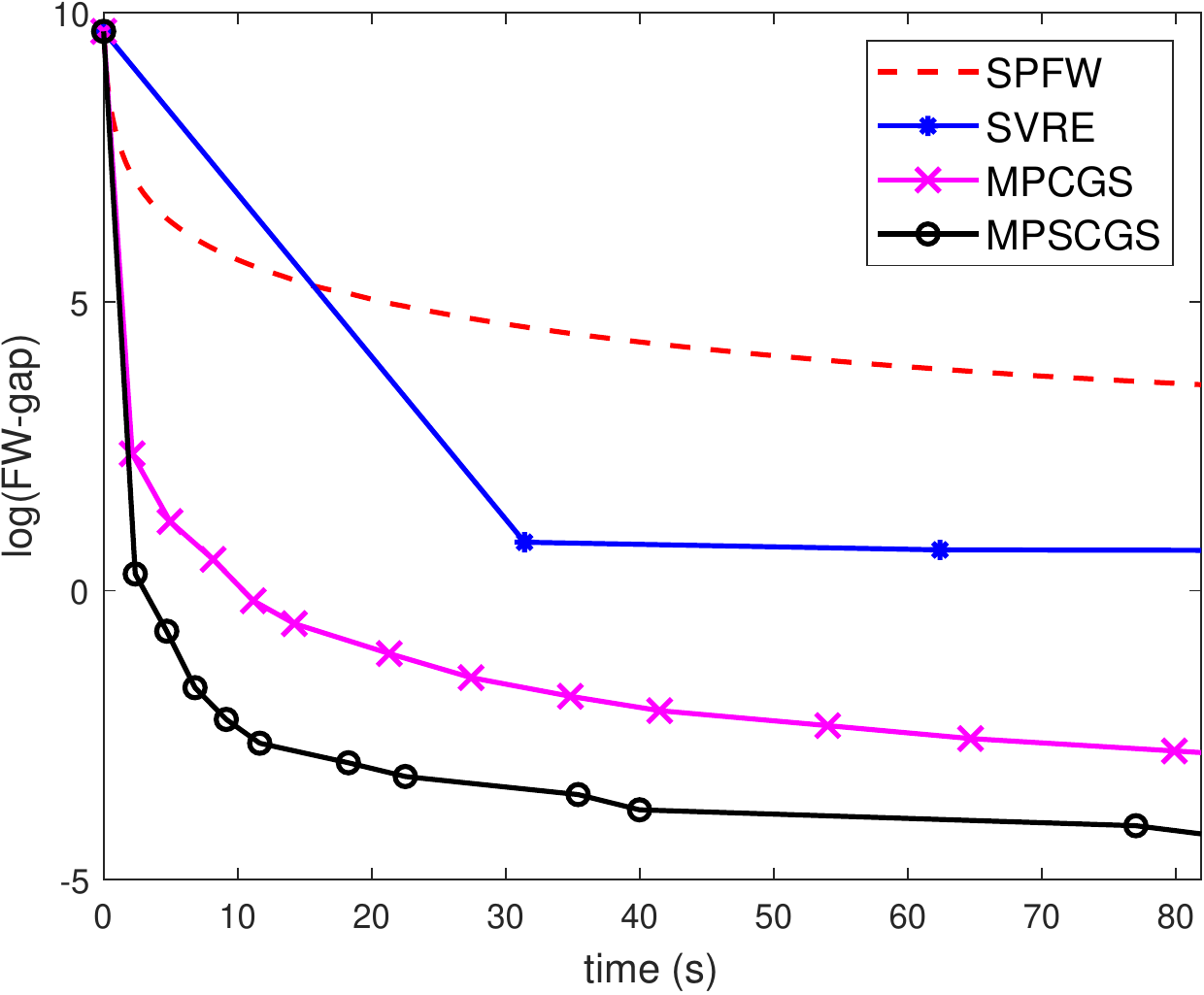} \\
			(a) rcv1 & (b) sector & (c) news20 \\[0.15cm]
		\end{tabular}
		\caption{We demonstrate the perfomance of algorithms by time (s) versus $\log(\text{FW-gap})$ for robust multiclass classification with nuclear norm ball constraint on datasets ``rcv1'', ``sector'', and ``news20''.}\label{fig:exp}
	\end{figure}

	\section{Conclusion and Future Works} \label{sec:con}
	In this paper, we propose projection-free algorithms for solving saddle point problems with complicated constraints in both batch and stochastic settings. Our methods are purely projection-free and do not require that the saddle point problem has special structures. We also provide convergence analysis for our algorithms in the convex-strongly-concave case. The experimental results demonstrate the effectiveness of our algorithms on three real world data sets. On the other hand, we believe that there is room for improving the complexity of the LO oracles, which will be our future studies. In addition, we will investigate how to extend our framework to the general convex-concave case and establish stronger convergence results in the strongly-convex-strongly-concave case.
	
	\section*{Broader Impact}
	This paper studies projection-free algorithms for convex-strongly-concave saddle point problems. From a theoretical viewpoint, our method propose the first stochastic projection-free algorithm for  saddle point problems without special conditions on the problem. From a practical viewpoint, our method can be applied to many machine learning applications which solve minimax problem with complicated constraints, e.g. robust optimization, matrix completion, two-player games and much more.
	
	%The proposed methods are general frameworks for optimization, rather than specific applications. Thus we think the discussion of broader impact is not applicable for our paper.

	\begin{ack}
		The team is supported by "New Generation of AI 2030" Major Project (2018AAA0100900) and National Natural Science Foundation of China (61702327, 61772333, 61632017). Luo Luo is supported by GRF 16201320.
	\end{ack}

	{\small\bibliographystyle{plainnat}
		\bibliography{reference}}

	\newpage
	\appendix

	\section{Proofs For MPCGS}
	In this section, we assume $f(\x,\y)$ satisfies Assumption \ref{ass1}.
	
	\subsection{Definitions and Lemmas}\label{appendix:dfn}
	We define the following functions:
	$$\y^*(\x)=\argmax_{\y\in\cY} f(\x,\y), $$
	$$\psi_k(\x)=(1-\gamma_k)\x_{k-1}+\gamma_k\cdot\argmin_{\v\in\cX}\left\{\nabla_\x f(\z_k,\y^*(\x))^\top\v+\frac{\alpha_k}{2}\|\v-\v_{k-1}\|^2\right\}. $$
	Since $f(\x,\cdot)$ is $\mu$-strongly-concave, $\y^*(\x)$ is unique. Then, we have the following two lemmas.
	
	\begin{lemma}[{\cite[Lemma 4.3]{lin2019gradient}}] \label{lem:kappa}
		$\y^*(\x)$ is $\kappa$-Lipschitz continuous.
	\end{lemma}

	\begin{lemma} \label{lem:contraction}
		$\psi_k(\x)$ is a $\frac{1}{2}$-contraction.
	\end{lemma}
	
	\begin{proof}
		Let 
		$$\nabla_1=\nabla_\x f(\z_k,\y^*(\x_1)), \quad \v_1=\argmin_{\v\in\cX}\left\{\nabla_1^\top\v+\frac{\alpha_k}{2}\|\v-\v_{k-1}\|^2\right\};$$ 
		$$\nabla_2=\nabla_\x f(\z_k,\y^*(\x_2)), \quad \v_2=\argmin_{\v\in\cX}\left\{\nabla_2^\top\v+\frac{\alpha_k}{2}\|\v-\v_{k-1}\|^2\right\}.$$
		Then we have
		$$\|\psi_k(\x_1)-\psi_k(\x_2)\|=\gamma_k\|\v_1-\v_2\|.$$
		According to the optimality of $\v_1$ and $\v_2$, we have
		$$\langle\nabla_1+\alpha_k(\v_1-\v_{k-1}),\v_1-\v_2\rangle\leq 0; $$
		$$\langle\nabla_2+\alpha_k(\v_2-\v_{k-1}),\v_2-\v_1\rangle\leq 0. $$
		Summing over inequalities, we get
		$$\langle\nabla_1-\nabla_2, \v_1-\v_2\rangle + \alpha_k\|\v_1-\v_2\|^2\leq 0 $$
		Thus,
		$$\|\v_1-\v_2\|^2\leq -\langle\nabla_1-\nabla_2, \v_1-\v_2\rangle/\alpha_k\leq \frac{1}{2}\|\v_1-\v_2\|^2+\frac{1}{2}\|\nabla_1-\nabla_2\|^2/\alpha_k^2,$$
		which indicates that
		$$\|\v_1-\v_2\|\leq \frac{\|\nabla_1-\nabla_2\|}{\alpha_k}\overset{\text{(a)}}{\leq} \frac{L\|\y^*(\x_1)-\y^*(\x_2)\|}{\alpha_k}\overset{\text{(b)}}{\leq} \frac{\kappa L\|\x_1-\x_2\|}{\alpha_k}$$
		Here (a) follows from the $L$-smoothness of $f$ and (b) follows from Lemma \ref{lem:kappa}.\\
		Since $\alpha_k = \frac{6\kappa_y L}{k+1}$ and $\gamma_k=\frac{3}{k+2}$, we have 
		$$\|\psi_k(\x_1)-\psi_k(\x_2)\|\leq \frac{\gamma_k\kappa L\|\x_1-\x_2\|}{\alpha_k}\leq \frac{1}{2}\|\x_1-\x_2\|.$$
	\end{proof}

	\begin{lemma}
		Assume the input parameters of the procedure Prox-step (Algorithm \ref{alg:prox}) is choosed as follows:
		$$\gamma=\frac{3}{k+2},\ \  \alpha = \frac{6\kappa L}{k+1}, \ \  \zeta = \frac{LD_\cX^2}{384k(k+1)}, \ \  \epsilon=\frac{\kappa LD_\cX^2}{k(k+1)(k+2)}.$$
		Then the Prox-step returns $(\x_R,\y_R,\v_R)$ which satisfies  $$f(\x_R,\y_R)\geq \max_{\y\in\cY}f(\x_R,\y)-\epsilon.$$
	\end{lemma}
	
	\begin{proof}
		Let $\psi(\x)=(1-\gamma)\x_0+\gamma \argmin_{\u\in\cX}\{\nabla_\x f(\z,\y^*(\x_r))^\top\u+\frac{\alpha}{2}\|\v-\u\|^2\}$.
		According to Lemma \ref{lem:contraction}, $\psi(\x)$ is a $\frac{1}{2}$-contraction.
		%Let $\v^*_r = \argmin_{\u\in\cX}\{\nabla_\x f(\z,\y^*(\x_r))^\top\u+\frac{\alpha}{2}\|\v-\u\|^2\}$, then $$\psi(\x_r)=(1-\gamma)\x_0+\gamma\v^*_r.$$
		
		By the optimality, we get
		\begin{equation} \label{eq:6}
		\langle\nabla_\x f(\z,\y^*(\x_{r-1}))+\alpha(\v^*_{r-1}-\v),\v^*_{r-1}-\v_r\rangle\leq 0.
		\end{equation}
		Since $\v_r=\text{CndG}(\nabla_\x f(\z,\y_r),\v,\alpha,\zeta,\cX)$, we have
		\begin{equation} \label{eq:7}
		\langle\nabla_\x f(\z,\y_r)+\alpha(\v_r-\v,\v_r-\v^*_{r-1}\rangle\leq \zeta.
		\end{equation}
		Sum Eq.(\ref{eq:6}) and Eq.(\ref{eq:7}) together, we have
		$$\langle \nabla_\x f(\z,\y^*(\x_{r-1}))-\nabla_\x f(\z,\y_r),\v^*_{r-1}-\v_r\rangle+\alpha \|\v^*_{r-1}-\v_r\|^2\leq\zeta. $$
		Thus, we can get
		\begin{equation*}
		\begin{split}
		\|\v^*_{r-1}-\v_r\|^2\leq& \frac{\zeta}{\alpha} - \frac{1}{\alpha}\langle \nabla_\x f(\z,\y^*(\x_{r-1}))-\nabla_\x f(\z,\y_r),\v^*_{r-1}-\v_r\rangle \\
		\leq& \frac{\zeta}{\alpha}+\frac{1}{2}\|\v^*_{r-1}-\v_r\|^2+\frac{\|\nabla_\x f(\z,\y^*(\x_{r-1}))-\nabla_\x f(\z,\y_r)\|^2}{2\alpha^2},
		\end{split}
		\end{equation*}
		which means 
		$$\|\v^*_{r-1}-\v_r\|^2 \leq \frac{2\zeta}{\alpha}+\frac{\|\nabla_\x f(\z,\y^*(\x_{r-1}))-\nabla_\x f(\z,\y_r)\|^2}{\alpha^2}.$$
		By the $L$-smoothness of $f$, we have
		\begin{equation} \label{eq:8}
		\begin{split}
		\|\v^*_{r-1}-\v_r\|\leq& \sqrt{\frac{L^2}{\alpha^2}\|\y^*(\x_{r-1})-\y_r\|^2+\frac{2\zeta}{\alpha}}\\
		\overset{\text{(a)}}{\leq}& \sqrt{\frac{2\kappa L}{\alpha^2}(f(\x_{r-1},\y^*(
			x_{r-1}))-f(\x_{r-1},\y_r))+\frac{2\zeta}{\alpha}}\\
		\overset{\text{(b)}}{\leq}& \sqrt{\frac{2\kappa L\epsilon_{cgs}}{\alpha^2}+\frac{2\zeta}{\alpha}}
		\end{split}
		\end{equation}
		where (a) is by the $\mu$-strongly-concavity of $f(\x,\cdot)$ and (b) is by the stopping condition of CGS.\\
		Assume the fix point of $\psi(\cdot)$ is $\tilde{\x}$. Then we bound  $\|\x_r-\tilde{\x}\|$ as follows:
		\begin{equation*}
		\begin{split}
		\|\x_r-\tilde{\x}\|=& \|(1-\gamma)\x_{k-1}+\gamma\v_r-\psi(\tilde{\x})\|\\
		\leq& \|\psi(\x_{r-1})-\psi(\tilde{\x})\|+\|(1-\gamma)\x_{k-1}+\gamma\v_{r}-(1-\gamma)\x_{k-1}-\gamma\v^*_{r-1}\|\\
		\leq& \frac{1}{2}\|\x_{r-1}-\tilde{\x}\|+\gamma\|\v_{r}-\v^*_{r-1}\|\\
		\leq& \frac{1}{2}\|\x_{r-1}-\tilde{\x}\|+\frac{\epsilon_{mp}}{4}\\
		\leq& 2^{-r}\|\x_0-\tilde{\x}\|+\frac{\epsilon_{mp}}{2}
		\end{split}
		\end{equation*}
		where the second inequality follows from Lemma \ref{lem:contraction}.
		Since $R=\left\lceil\log_2\frac{4D_\cX}{\epsilon_{mp}} \right\rceil$, we know that 
		$$\|\x_{R-1}-\tilde{\x}\|\leq 2\cdot 2^{-R}\|\x_0-\tilde{\x}\|+\frac{\epsilon_{mp}}{2}\leq 2\cdot\frac{\epsilon_{mp}}{4}+\frac{\epsilon_{mp}}{2}= \epsilon_{mp}. $$
		Then, we can get
		\begin{equation*}
		\begin{split}
		\|\y_R-\y^*(\x_{R})\|\leq& \|\y^*(\x_{R})-\y^*(\tilde{\x})\|+\|\y^*(\x_{R-1})-\y^*(\tilde{\x})\|+\|\y^*(\x_{R-1})-\y_R\|\\
		\leq& \kappa(\|\x_{R}-\tilde{\x}\|+\|\x_{R-1}-\tilde{\x}\|)+\sqrt{\frac{2}{\mu}\epsilon_{cgs}}\\
		\leq& 2\kappa\epsilon_{mp}+\sqrt{\frac{2}{\mu}\epsilon_{cgs}}\\
		=& 8\kappa\gamma_k\sqrt{\frac{2\kappa L\epsilon_{cgs}}{\alpha_k^2}+\frac{2\zeta_k}{\alpha_k}}+\sqrt{\frac{2}{\mu}\epsilon_{cgs}}
		\end{split}
		\end{equation*}
		
		where the second inequality comes from Lemma \ref{lem:kappa} and the concavity of $f(\x,\cdot)$.
		According to the value of input parameters, we can get
		$$\|\y_R-\y^*(\x_{R})\| \leq \sqrt{\frac{\kappa D_\cX^2}{k(k+1)(k+2)}}+\sqrt{\frac{\kappa D_\cX^2}{32k(k+1)(k+2)}}\leq \sqrt{\frac{2\kappa D_\cX^2}{k(k+1)(k+2)}}. $$
		
		By the $L$-smoothness of $f$ and the optimality of $\y^*(\x_R)$, we have
		\begin{equation*}
		\begin{split}
		f(\x_R,\y_R)\geq& f(\x_R,\y^*(\x_R))+\langle\nabla_\y f(\x_R,\y^*(\x_R)),\y_R-\y^*(\x_R)\rangle-\frac{L}{2} \|\y_R-\y^*(\x_{R})\|^2\\
		\geq& f(\x_R,\y^*(\x_R))-\frac{L}{2} \|\y_R-\y^*(\x_{R})\|^2\\
		\geq& f(\x_R,\y^*(\x_R))-\frac{\kappa L D_\cX^2}{k(k+1)(k+2)}=f(\x_R,\y^*(\x_R))-\epsilon.
		\end{split}
		\end{equation*}
		
	\end{proof}
	In addition, by the updating formula of Prox-step, it is obvious that the Line 4 of Algorithm \ref{alg:mpcgs} ensures that
	$\v_k= \mathrm{CndG}(\nabla_\x f(\z_k,\y_k),\v_{k-1},\alpha_k,\zeta_k,\cX)$ and $\x_k=(1-\gamma_k)\x_{k-1}+\gamma_k\v_k$.

	\subsection{Proof of Theorem \ref{thm:csc}}
	\begin{proof}
		According to the smoothness, for any $\tilde{\x}\in\cX$, we have
		{\small\begin{equation}
			\begin{split}\label{eq:1}
			& f(\x_k,\y_k) \\
			\leq  &  f(\z_k,\y_k)+\nabla_\x f(\z_k,\y_k)^\top(\x_k-\z_k)+ \frac{L}{2}\|\x_k-\z_k\|^2\\
			=& (1-\gamma_k)( f(\z_k,\y_k)+\nabla_\x f(\z_k,\y_k)^\top(\x_{k-1}-\z_k))+\gamma_k(f(\z_k,\y_k)+\nabla_\x f(\z_k,\y_k)^\top(\tilde{\x}-\z_k))\\
			& +\gamma_k\nabla_\x f(\z_k,\y_k)^\top(\v_k-\tilde{\x})+\frac{\gamma_k^2L}{2}\|\v_k-\v_{k-1}\|^2\\
			\leq& (1-\gamma_k) f(\x_{k-1},\y_k)+\gamma_k f(\tilde{\x},\y_k)+\gamma_k\nabla_\x f(\z_k,\y_k)^\top(\v_k-\tilde{\x})+\frac{\gamma_k^2L}{2}\|\v_k-\v_{k-1}\|^2.
			\end{split}
			\end{equation}}
		
		where the last inequality comes from the convexity of $f(\cdot,\y_k)$. Notice that the update of $\v_k$ and the stopping condition of CndG procedure ensures that
		
		\begin{equation} \label{eq:2}
		\max_{\x\in\cX} \langle \nabla_\x f(\z_k,\y_k)+\alpha_k(\v_k-\v_{k-1}), \v_k-\x \rangle\leq \zeta_k.
		\end{equation} 
		
		Combining Eq.(\ref{eq:1}) and (\ref{eq:2}) we can get
		{\small\begin{equation}
			\begin{split}
			& f(\x_k,\y_k)-f(\tilde{\x},\y_k)\\
			\leq& (1-\gamma_k)(f(\x_{k-1},\y_k)-f(\tilde{\x},\y_k))+\gamma_k\zeta_{k}-\gamma_k \alpha_k(\v_k-\v_{k-1})^\top(\v_k-\tilde{\x})+\frac{\gamma_k^2L}{2}\|\v_k-\v_{k-1}\|^2\\
			=& (1-\gamma_k)(f(\x_{k-1},\y_k)-f(\tilde{\x},\y_k))+\gamma_k\zeta_{k}+\frac{\gamma_k \alpha_k}{2}\left(\|\v_{k-1}-\tilde{\x}\|^2-\|\v_k-\tilde{\x}\|^2\right)\\
			&+\frac{\gamma_k}{2}(L\gamma_k- \alpha_k)\|\v_k-\v_{k-1}\|^2.
			\end{split}
			\end{equation}}

		Let $\Phi(k)=k(k+1)(k+2)(f(\x_k,\y_k)-f(\tilde{\x},\y_k))$. According to line \ref{line:mpcgs-sub} of Algorithm \ref{alg:mpcgs}, we have $$f(\x_k,\y_k)\geq \max_{\y\in\cY}f(\x_k,\y)-\epsilon_k.$$
		Thus, we have
		{\small\begin{equation} \label{eq:5}
			\begin{split}
			\Phi(k)\leq& \Phi(k-1)+k(k-1)(k+1)(f(\x_{k-1},\y_k)-f(\x_{k-1},\y_{k-1})-f(\tilde{\x},\y_k)+f(\tilde{\x},\y_{k-1}))\\
			&+k(k+1)(k+2)\left(\gamma_k\eta_{t}+\frac{\gamma_k \alpha_k}{2}(\|\v_{k-1}-\tilde{\x}\|^2-\|\v_k-\tilde{\x}\|^2)\right)\\
			\leq& \Phi(k-1)+k(k-1)(k+1)(\epsilon_{k-1}-f(\tilde{\x},\y_k)+f(\tilde{\x},\y_{k-1}))\\
			&+k(k+1)(k+2)\left(\gamma_k\zeta_k+\frac{\gamma_k \alpha_k}{2}(\|\v_{k-1}-\tilde{\x}\|^2-\|\v_k-\tilde{\x}\|^2)\right)\\
			\leq& \Phi(0)+\sum_{s=1}^k s(s-1)(s+1)\epsilon_{s-1}-\sum_{s=1}^k s(s-1)(s+1)(f(\tilde{\x},\y_s)-f(\tilde{\x},\y_{s-1}))\\
			&+ \sum_{s=1}^k s(s+1)(s+2)\left(\gamma_s\zeta_{s}+\frac{\gamma_s \alpha_s}{2}(\|\v_{s-1}-\tilde{\x}\|^2-\|\v_s-\tilde{\x}\|^2)\right)
			\end{split}
			\end{equation}}
		Notice that 
		\begin{equation} \label{eq:3}
		\sum_{s=1}^k s(s+1)(s+2)\gamma_s\zeta_{s}=\frac{1}{128}k LD_\cX^2
		\end{equation}
		and
		\begin{equation} \label{eq:4}
		\begin{split}
		&\sum_{s=1}^k s(s+1)(s+2)\frac{\gamma_s \alpha_s}{2}(\|\v_{s-1}-\tilde{\x}\|^2-\|\v_s-\tilde{\x}\|^2)\\
		=& 9\kappa L\sum_{s=1}^k s(\|\v_{s-1}-\tilde{\x}\|^2-\|\v_s-\tilde{\x}\|^2)\\
		\leq& 9\kappa L\sum_{s=0}^{k-1} \|\v_s-\tilde{\x}\|^2\leq 9k\kappa LD_\cX^2.
		\end{split}
		\end{equation}
		Substituting  Eq.(\ref{eq:3}) and Eq.(\ref{eq:4}) into Eq.(\ref{eq:5}) and using the fact that $\Phi(0)=0$, we have
		
		\begin{equation*}
		\Phi(k)\leq \sum_{s=1}^k s(s-1)(s+1)\epsilon_{s-1}-\sum_{s=1}^k s(s-1)(s+1)(f(\tilde{\x},\y_s)-f(\tilde{\x},\y_{s-1}))+ 10k\kappa LD_\cX^2.
		\end{equation*}
		
		Thus, for any $\tilde{\y}\in\cY$, we have
		
		\begin{equation*}
		\begin{split}
		& \sum_{s=1}^k s(s-1)(s+1)\epsilon_{s-1} + 10k\kappa LD_\cX^2\\
		\geq& \Phi(k)+\sum_{s=1}^k s(s-1)(s+1)(f(\tilde{\x},\y_s)-f(\tilde{\x},\y_{s-1}))\\
		\geq& \Phi(k)-3\sum_{s=1}^{k-1} s(s+1)f(\tilde{\x},\y_s)+k(k-1)(k+1)f(\tilde{\x},\y_k)\\
		=& k(k+1)(k+2)f(\x_k,\y_k)-3\sum_{s=1}^{k} s(s+1)f(\tilde{\x},\y_s)\\
		\overset{\text{(a)}}{\geq}& k(k+1)(k+2)f[f(\x_k,\y_k)-f(\tilde{\x},\bar{\y}_k)]\\
		\geq& k(k+1)(k+2)[f(\x_k,\tilde{\y})-f(\tilde{\x},\bar{\y}_k)-\epsilon_k]
		\end{split}
		\end{equation*}
		where (a) is by the concavity of $f(\x,\cdot)$ and the definition of $\bar{\y}_k$. Then, we can obtain the bound of the primal-dual gap:
		\begin{equation*}
		f(\x_k,\tilde{\y})-f(\tilde{\x},\bar{\y}_k)\leq \frac{1}{k(k+1)(k+2)}\left(\sum_{s=1}^k s(s+1)(s+2)\epsilon_s+10k\kappa LD_\cX^2  \right)\leq \frac{11\kappa LD_\cX^2}{(k+1)(k+2)}
		\end{equation*}
		where the last equation comes from the fact that $\epsilon_s=\frac{\kappa LD_\cX^2}{s(s+1)(s+2)}$.
		
	\end{proof}
	
	\subsection{Proof of Corollary \ref{col:csc}}
	According to Theorem 2.5 of \cite{lan2016conditional}, the CGS method requires  $\cO\left(\sqrt{\kappa}\log\frac{\kappa LD_\cY^2}{\epsilon_{k}}\right)$ FO calls and $\cO\left(\frac{\kappa LD_\cY^2}{\epsilon_{k}}\right)$ LO calls, where $R$ is the number of iteration of the Prox-step procedure. Thus, the FO and LO complexity of MPCGS are respectively 
	$\cO\left(\sum_{k=1}^N\sum_{r=1}^R \sqrt{\kappa}\log\frac{\kappa LD_\cY^2}{\epsilon_k}\right) $ and
	$\cO\left(\sum_{k=1}^N\sum_{r=1}^R \left(\frac{\kappa LD_\cY^2}{\epsilon_k}+\frac{\alpha_k D_\cX^2}{\zeta_k}\right)\right)$.
	Theorem \ref{thm:csc} implies that $N$ should be the order $\Theta\left(\sqrt{\frac{\kappa LD_\cX^2}{\epsilon}}\right)$. Plugging in all parameters obtains the complexity of MPCGS.

	\section{Proof of Inexact STORC Algorithm}
	In this section, we provide the details for the analysis of iSTORC (Algorithm \ref{alg:istorc}). We suppose that the objective function $h(\x)$ satisfies assumptions in Section 4.1. We define $D_t=\sqrt{\frac{\kappa D^2}{2^{t-1}}}$ for any $t\in[N]$.

	\subsection{Technical Lemmas}
	\begin{lemma}[Lemma 3 of \cite{hazan2016variance}] \label{lem:istorc}
		At the $t$-th outer iteration of iSTORC (Algorithm \ref{alg:istorc}), we suppose that $\BE[\|\x_0-\x^*\|]\leq D_t^2$. Then for any $k$, we have
		$$ \BE[h(\x_k)-h(\x^*)]\leq \frac{8LD_t^2}{k(k+1)} $$
		if $\BE[\|\r_s-\nabla h(\w_s)\|^2]\leq \frac{L^2D_t^2}{Ms(s+1)}$ for all $s\leq k$.
	\end{lemma}

	% \begin{lemma} \label{lem:var}
	% In the finite-sum setting, we have
	% $$\BE[\|\nabla h_{\cQ_t}(\x)-\nabla h(\x)\|^2]=0.$$
	% In the expectation setting, we have
	% $$\BE[\|\nabla h_{\cQ_t}(\x)-\nabla h(\x)\|^2]\leq \frac{L^2D^2_t}{2M^2(M+1)}.$$
	% \end{lemma}
	% \begin{proof}
	% In the finite-sum setting, since $\nabla h_{\cQ_t}(\x)$ takes the full batch, the variance is zero.\\
	% In  expectation setting, we have
	% $$\BE[\|\nabla h_{\cQ_t}(\x)-\nabla h(\x)\|^2]\leq \frac{\sigma^2}{Q_t}\leq \frac{L^2D^2_t}{4M^2(M+1)}.$$
	% \end{proof}

	\begin{lemma} \label{lem:var}
		For iSTORC, we have
		$$\BE[\|\nabla h_{\cQ_t}(\x)-\nabla h(\x)\|^2]\leq \frac{L^2D^2_t}{4M^2(M+1)}.$$
	\end{lemma}
	\begin{proof}
		In  expectation setting, we have
		$$\BE[\|\nabla h_{\cQ_t}(\x)-\nabla h(\x)\|^2]\leq \frac{\sigma^2}{Q_t}\leq \frac{L^2D^2_t}{4M^2(M+1)}.$$
	\end{proof}
	
	\subsection{Proof of Theorem \ref{thm:istorc}}
	\begin{proof}
		We prove by this theorem by induction.
		For $t=0$, by the smoothness and the convexity of $h$, we have
		$$h(\bar{\x}_0)\leq h(\x^*)+\nabla h(\x^*)^\top(\bar{\x}_0-\x^*)+\frac{L}{2}\|\bar{\x}_0-\x^*\|^2\leq h(\x^*)+\frac{LD^2}{2}.$$
		Then, we suppose $\BE[h(\bar{\x}_{t-1})-h(\x^*)]\leq \frac{LD^2}{2^t}$ and consider the $t$-th outer iteration.\\
		By the strong convexity and the inductive assumption, we have
		$$\BE[\|\x_0-\x^*\|^2]\leq \frac{2}{\mu}\BE[h(\bar{\x}_{t-1})-h(\x^*)]\leq \frac{LD^2}{\mu 2^{t-1}}=D_t^2. $$
		
		Now we use another induction on the inner iteration to show that it holds
		$\BE[h(\x_k)-h(\x^*)]\leq \frac{8LD_t^2}{k(k+1)} $
		for any $1\leq k \leq M$.\\
		\\
		For $k=1$, by the fact that $\w_1=\x_0$ and Lemma \ref{lem:var},  we have
		$$\BE[\|\r_1-\nabla h(\w_1)\|^2]= \BE[\|\nabla h_{\cQ_t}(\x_0)-\nabla h(\x_0)\|^2]\leq \frac{L^2D_t^2}{4M}$$
		for both finite-sum setting and expectation setting. Then, by Lemma \ref{lem:istorc} we can obtain
		$$\BE[h(\x_1)-h(\x^*)]\leq 4LD_t^2. $$
		
		Now we suppose $\BE[h(\x_s)-h(\x^*)]\leq \frac{8LD_t^2}{s(s+1)} $ for $s<k$ where $k\geq 2$. We consider the case $s=k$.
		
		Since  the variance of a variable is less than its second-order moment, we can get
		{\small\begin{equation} \label{eq:9}
			\begin{split}
			& \BE_{\cS_{t,s},\cQ_t}[\|\r_s-\nabla h(\w_s)\|^2]\\
			=& \BE_{\cS_{t,s},\cQ_t}[\|\nabla h_{\cS_{t,s}}(\w_s)-\nabla h_{\cS_{t,s}}(\x_0)+\nabla h_{\cQ_t}(\x_0)-\nabla h(\w_s)\|^2]\\
			\leq& 2\BE_{\cS_{t,s}}[\|\nabla h_{\cS_{t,s}}(\w_s)-\nabla h_{\cS_{t,s}}(\x_0)-(\nabla h(\w_s)-\nabla h(\x_0))\|^2]+2\BE_{\cQ_t}[\|\nabla h_{\cQ_t}(\x_0)-\nabla h(\x_0)\|^2]\\
			\leq& 2\BE_{\cS_{t,s}}[\|\nabla h_{\cS_{t,s}}(\w_s)-\nabla h_{\cS_{t,s}}(\x_0)\|^2]+2\BE_{\cQ_t}[\|\nabla h_{\cQ_t}(\x_0)-\nabla h(\x_0)\|^2]\\
			\leq& \frac{2L^2}{S}\|\w_s-\x_0\|^2+2\BE_{\cQ_t}[\|\nabla h_{\cQ_t}(\x_0)-\nabla h(\x_0)\|^2].
			\end{split}
			\end{equation}}
		Since 
		$ \|\w_s-\x_0\|^2\leq 2\|\w_s-\x^*\|^2 + 2\|\x_0-\x^*\|^2$, then we bound $\BE[\|\x_0-\x^*\|^2]$ and $\BE[\|\w_s-\x^*\|^2]$ separately.
		
		For $\BE[\|\x_0-\x^*\|^2]$, we have
		\begin{equation} \label{eq:15}
		\BE[\|\x_0-\x^*\|^2]\leq D_t^2=\frac{2\kappa D_t^2}{2\kappa}\leq \frac{64\kappa D_t^2}{M(M+1)}.   
		\end{equation}

		Since $\w_s=(1-\lambda_s)\x_{s-1}+\lambda_s\u_{s-1}$ and $\x_{s-1}=(1-\lambda_{s-1})\x_{s-2}+\lambda_{s-1}\u_{s-1}$, we have
		$$\w_s=\frac{\lambda_{s-1}+\lambda_s-\lambda_{s-1}\lambda_s}{\lambda_{s-1}}\x_{s-1}-\frac{\lambda_s-\lambda_{s-1}\lambda_s}{\lambda_{s-1}}\x_{s-2}$$
		Thus, we have
		\begin{equation} \label{eq:14}
		\begin{split}
		\|\w_s-\x^*\|^2=& \left\|\frac{\lambda_{s-1}+\lambda_s-\lambda_{s-1}\lambda_s}{\lambda_{s-1}}(\x_{s-1}-\x^*)-\frac{\lambda_s-\lambda_{s-1}\lambda_s}{\lambda_{s-1}}(\x_{s-2}-\x^*)\right\|^2\\
		\leq& 2\left\|\frac{\lambda_{s-1}+\lambda_s-\lambda_{s-1}\lambda_s}{\lambda_{s-1}}(\x_{s-1}-\x^*)\right\|^2+2\left\|\frac{\lambda_s-\lambda_{s-1}\lambda_s}{\lambda_{s-1}}(\x_{s-2}-\x^*)\right\|^2\\
		\leq& 8\|\x_{s-1}-\x^*\|^2+2\|\x_{s-2}-\x^*\|^2\\
		\end{split}
		\end{equation}
		where the last inequality comes from the fact $\lambda_s\leq\lambda_{s-1}\leq 1$.
		
		If $s=2$, we can obtain 
		\begin{equation*}
		\begin{split}
		L^2\BE[\|\w_2-\x_0\|^2]\leq& L^2\BE[16\|\x_1-\x^*\|^2+6\|\x_0-\x^*\|^2]\\
		\leq& 32\kappa L\BE[h(\x_1)-h(\x^*)]+\frac{384\kappa L^2D_t^2}{(M+1)^2}\\
		\leq& \frac{768\kappa L^2D_t^2}{s(s+1)} + \frac{384\kappa L^2D_t^2}{s(s+1)}= \frac{1152\kappa L^2D_t^2}{s(s+1)}
		\end{split}
		\end{equation*}
		where the second inequality comes from the strong convexity of $h(\x)$ and Eq. (\ref{eq:15}), the third inequality comes from the induction hypothesis.
		
		If $s\geq 3$, we can obtain
		\begin{equation} \label{eq:16}
		\begin{split}
		L^2\BE[\|\w_s-\x^*\|^2]\leq& 8L^2\BE[\|\x_{s-1}-\x^*\|^2]+2L^2\BE[\|\x_{s-2}-\x^*\|^2]\\
		\leq& 16\kappa L\BE[h(\x_{s-1})-h(\x^*)]+4\kappa L\BE[h(\x_{s-2})-h(\x^*)]\\
		\leq& \frac{128\kappa L^2D_t^2}{s(s-1)}+\frac{32\kappa L^2D_t^2}{(s-2)(s-1)}\leq \frac{448\kappa L^2D_t^2}{s(s+1)}
		\end{split}
		\end{equation}
		which indicates that  
		$$L^2\BE[\|\w_s-\x_0\|^2]\leq L^2\BE[2\|\w_s-\x^*\|^2 + 2\|\x_0-\x^*\|^2]\leq \frac{1024\kappa L^2D_t^2}{s(s+1)} $$
		where the last inequality is due to Eq.(\ref{eq:15}) and Eq.(\ref{eq:16}).
		Substituting all these pieces into Eq.(\ref{eq:9}), we can obtain
		\begin{equation*}
		\begin{split}
		\BE_{\cS_{t,s},\cQ_t}[\|\r_s-\nabla h(\w_s)\|^2] \leq&  \frac{2L^2}{S}\|\w_s-\x_0\|^2+2\BE_{\cQ_t}[\|\nabla h_{\cQ_t}(\x_0)-\nabla h(\x_0)\|^2]\\
		\leq& \frac{2400\kappa L^2D_t^2}{s(s+1)S}+\frac{2\sigma^2}{Q_t}\\
		\leq& \frac{L^2D_t^2}{2Ms(s+1)}+\frac{L^2D_t^2}{2Ms(s+1)}=\frac{L^2D_t^2}{Ms(s+1)}
		\end{split}
		\end{equation*}
		where the last inequality comes from the fact that $s\leq M$.
		By Lemma \ref{lem:istorc}, we know that $\BE[h(\x_s)-h(\x^*)]\leq \frac{8LD_t^2}{s(s+1)}$, which completes that induction.
	\end{proof}
	
	\subsection{Proof of Corollary \ref{col:istorc}}
	\begin{proof}
		Theorem \ref{thm:istorc} implies that $N$ should be the order $\Theta\left(\log_2{\frac{LD^2}{\epsilon}}\right)$. Thus, the SFO complexity of iSTORC is $$\cO\left(\sum_{t=1}^N\left(Q_t+\sum_{k=1}^M S\right)\right)=\cO\left(\left(\frac{\sqrt{\kappa}}{L\epsilon}+\kappa^2\right)\log_2\left(\frac{LD^2}{\epsilon}\right)\right).$$ Since the CndG procedure requires $\cO\left(\frac{\beta_kD^2}{\eta_{t,k}}\right)$ iterations, the LO complexity of iSTORC is $\cO\left(\frac{LD^2}{\epsilon}\right)$.
	\end{proof}
	
	\section{Proof for MPSCGS}
	In this section, we assume $f(\x,\y)$ satisfies Assumption \ref{ass1} and Assumption \ref{ass2}. We also use the same definition of $\y^*(\cdot)$ and $\psi_k(\cdot)$ as in Section \ref{appendix:dfn}. We denote $\sigma_k^2=\frac{\sigma^2}{P_k}$ additionally.

	\subsection{Notation and Lemma}
	
	\begin{lemma} \label{lem:sprox}
		We set the input of Stochastic-Prox-step (Algorithm \ref{alg:sprox}) as
		$$\gamma=\frac{3}{k+2},\ \  \alpha = \frac{6\kappa L}{k+1}, \ \  \zeta = \frac{LD_\cX^2}{576k(k+1)}, \ \  \epsilon=\frac{\kappa LD_\cX^2}{k(k+1)(k+2)},\ \  P=\left\lceil\frac{96\sigma^2(k+1)^3}{\kappa L^2D_\cX^2} \right\rceil, $$
		then the output $(\x_R,\y_R,\v_R)$ satisfies  $$\BE[f(\x_R,\y_R)]\geq \BE\left[\max_{\y\in\cY}f(\x_R,\y)\right]-\epsilon.$$
	\end{lemma}
	
	% \begin{lemma}
	% If we follows the parameters in Theorem \ref{thm:mpscgs}. In the $k$-th iteration of MPSCGS, the call of  procedure Stochastic-Prox-step (Algorithm \ref{alg:sprox}) returns $(\x_R,\y_R,\v_R)$ which satisfies
	% $$ \BE[f(\x_R,\y_R)]\geq \BE[\max_{\y\in\cY}f(\x_R,\y)]-\epsilon_k, $$
	% $$ \v_R= \mathrm{CndG}(\nabla_\x f(\z_k,\y_R),\v_{k-1},\alpha_k,\zeta_k,\cX),$$ $$\x_R=(1-\gamma_k)\x_{k-1}+\gamma_k\v_R.$$
	% \end{lemma}
	\begin{proof}
		Let $\psi(\x)=(1-\gamma)\x_0+\gamma \argmin_{\u\in\cX}\left\{\nabla_\x f(\z,\y^*(\x_r))^\top\u+\frac{\alpha}{2}\|\v-\u\|^2\right\}$.
		According to Lemma \ref{lem:contraction}, $\psi(\x)$ is a $\frac{1}{2}$-contraction.
		Let $\v^*_r = \argmin_{\u\in\cX}\{\nabla_\x f(\z,\y^*(\x_r))^\top\u+\frac{\alpha}{2}\|\v-\u\|^2\}$, then $$\psi(\x_r)=(1-\gamma)\x_0+\gamma\v^*_r.$$
		
		%Let $\v^*_r = \argmin_{\v\in\cX}\{\nabla_\x f(\z_k,\y^*(\x_r))^\top\v+\frac{\alpha_k}{2}\|\v-\v_{k-1}\|^2\}$, then $$\psi_k(\x_r)=(1-\gamma_k)\x_{k-1}+\gamma_k\v^*_r.$$
		%$$\v_r = \argmin_{\v\in\cX}\{\nabla_\x f(\w,\y_{r})^\top\v+\frac{\alpha_k}{2}\|\v-\v_{k-1}\|^2\}.$$
		By the optimality, we get
		\begin{equation} \label{eq:12}
		\langle\nabla_\x f(\z,\y^*(\x_{r-1}))+\alpha(\v^*_{r-1}-\v),\v^*_{r-1}-\v_r\rangle\leq 0.
		\end{equation}
		Since $\v_r=\text{CndG}(\nabla_\x f_{\cP}(\z,\y_r),\v,\alpha,\zeta,\cX)$, we have
		\begin{equation} \label{eq:13}
		\langle\nabla_\x f(\z,\y_r)+\alpha(\v_r-\v),\v_r-\v^*_{r-1}\rangle\leq \zeta
		\end{equation}
		Sum Eq.(\ref{eq:12}) and Eq.(\ref{eq:13}) together, we have
		$$\langle \nabla_\x f(\z,\y^*(\x_{r-1}))-\nabla_\x f_{\cP}(\z,\y_r),\v^*_{r-1}-\v_r\rangle+\alpha \|\v^*_{r-1}-\v_r\|^2\leq\zeta $$
		Thus, we can get
		\begin{equation*}
		\begin{split}
		\BE[\|\v^*_{r-1}-\v_r\|^2]\leq& \frac{\zeta}{\alpha}-\frac{1}{\alpha}\BE[\langle \nabla_\x f(\z,\y^*(\x_{r-1}))-\nabla_\x f_{\cP}(\z,\y_r),\v^*_{r-1}-\v_r\rangle]\\
		\leq& \frac{\zeta}{\alpha}+\frac{1}{2}\BE[\|\v^*_{r-1}-\v_r\|^2]+\frac{\BE[\|\nabla_\x f(\z,\y^*(\x_{r-1}))-\nabla_\x f_{\cP}(\z,\y_r)\|^2]}{2\alpha^2}
		\end{split}
		\end{equation*}
		which indicates
		\begin{equation*}
		\begin{split}
		&\BE[\|\v^*_{r-1}-\v_r\|^2]\\
		\leq& \frac{\BE[\|\nabla_\x f(\z,\y^*(\x_{r-1}))-\nabla_\x f_{\cP}(\z,\y_r)\|^2]}{\alpha^2}+\frac{2\zeta}{\alpha}\\
		\leq& \frac{\BE[2\|\nabla_\x f(\z,\y^*(\x_{r-1}))-\nabla_\x f(\z,\y_r)\|^2+2\|\nabla_\x f(\z,\y_r)-\nabla_\x f_{\cP}(\z,\y_r)\|^2]}{\alpha^2}+\frac{2\zeta}{\alpha}
		\end{split}
		\end{equation*}
		Denote $\widehat{\sigma}^2=\frac{\sigma^2}{P}$. By the $L$-smoothness of $f$, we can obtain
		
		\begin{equation*}
		\begin{split}
		\BE[\|\v^*_{r-1}-\v_r\|^2]\leq& \frac{2L^2}{\alpha^2}\BE[\|\y^*(\x_{r-1})-\y_r\|^2]+\frac{2\zeta}{\alpha}+\frac{2\widehat{\sigma}^2}{\alpha^2}\\
		\leq& \frac{4\kappa L}{\alpha^2}\BE[f(\x_{r-1},\y^*(
		x_{r-1}))-f_{\cP}(\x_{r-1},\y_r)]+\frac{2\zeta}{\alpha}+\frac{2\widehat{\sigma}^2}{\alpha^2}\\
		\leq& \frac{4\kappa L\epsilon_{cgs}}{\alpha^2}+\frac{2\zeta}{\alpha}+\frac{2\widehat{\sigma}^2}{\alpha^2}
		\end{split}
		\end{equation*}
		where we use the $\mu$-strongly-concavity of $f(\x,\cdot)$ and  the stopping condition of iSTORC. Then we bound $\BE[\|\x_r-\tilde{\x}\|^2]$ as follows:
		\begin{equation*}
		\begin{split}
		\BE[\|\x_r-\tilde{\x}\|^2]=& \BE[\|(1-\gamma)\z+\gamma\v_r-\psi(\tilde{\x})\|^2]\\
		\leq& 2\BE[\|\psi(\x_{r-1})-\psi(\tilde{\x})\|^2]+2\BE[\|(1-\gamma)\z+\gamma\v_{r}-(1-\gamma)\z-\gamma\v^*_{r-1}\|^2]\\
		\leq& \frac{1}{2}\BE[\|\x_{r-1}-\tilde{\x}\|^2]+2\gamma^2\BE[\|\v_{r}-\v^*_{r-1}\|^2]\\
		\leq& \frac{1}{2}\BE[\|\x_{r-1}-\tilde{\x}\|^2]+\frac{\epsilon_{mp}}{4}\\
		\leq& 2^{-r}\BE[\|\x_0-\tilde{\x}\|^2]+\frac{\epsilon_{mp}}{2}
		\end{split}
		\end{equation*}
		Since $R=\left\lceil\log_2\left(\frac{4D_\cX^2}{\epsilon_{mp}}\right) \right\rceil$, we know that 
		$$\BE[\|\x_{R-1}-\tilde{\x}\|^2]\leq 2\cdot 2^{-R}\BE[\|\x_0-\tilde{\x}\|^2]+\frac{\epsilon_{mp}}{2}\leq 2\cdot\frac{\epsilon_{mp}}{4}+\frac{\epsilon_{mp}}{2}= \epsilon_{mp} $$
		Then, we can get
		\begin{equation*}
		\begin{split}
		\BE[\|\y_R-\y^*(\x_{R})\|^2]\leq& 3\BE\left[\|\y^*(\x_{R})-\y^*(\tilde{\x})\|^2+\|\y^*(\x_{R-1})-\y^*(\tilde{\x})\|^2+\|\y^*(\x_{R-1})-\y_R\|^2\right]\\
		\leq& 3\kappa\left(\BE\left[\|\x_{R}-\tilde{\x}\|^2+\|\x_{R-1}-\tilde{\x}\|^2\right]\right)+\frac{6}{\mu}\epsilon_{cgs}\\
		\leq& 6\kappa\epsilon_{mp}+\frac{6}{\mu}\epsilon_{cgs}\\
		=& 48\kappa\gamma^2\left(\frac{4\kappa L\epsilon_{cgs}}{\alpha^2}+\frac{2\zeta}{\alpha}+\frac{2\sigma^2}{\alpha^2}\right)+\frac{6}{\mu}\epsilon_{cgs}
		\end{split}
		\end{equation*}
		where the second inequality comes from Lemma \ref{lem:kappa} and the concavity of $f(\x,\cdot)$.
		According to the value of input parameters, we can get
		\begin{align*}
		\BE[\|\y_R-\y^*(\x_{R})\|^2]\leq \frac{2\kappa D_\cX^2}{k(k+1)(k+2)}.
		\end{align*}
		
		By the $L$-smoothness of $f$ and the optimality of $\y^*(\x_R)$, we have
		\begin{equation*}
		\begin{split}
		\BE[f(\x_R,\y_R)]\geq& \BE[f(\x_R,\y^*(\x_R))+\langle\nabla_\y f(\x_R,\y^*(\x_R)),\y_R-\y^*(\x_R)\rangle-\frac{L}{2} \|\y_R-\y^*(\x_{R})\|^2]\\
		\geq& \BE[f(\x_R,\y^*(\x_R))-\frac{L}{2} \|\y_R-\y^*(\x_{R})\|^2]\\
		\geq& \BE\left[f(\x_R,\y^*(\x_R))-\frac{\kappa L D_\cX^2}{k(k+1)(k+2)}\right]=\BE[f(\x_R,\y^*(\x_R))]-\epsilon.
		\end{split}
		\end{equation*}
	\end{proof}

	\subsection{Proof of Theorem \ref{thm:mpscgs}}
	\begin{proof}
		Similar to Eq.(\ref{eq:1}) in section A.2, for any $\tilde{\x}\in\cX$, we have
		$$f(\x_k,\y_k)\leq (1-\gamma_k) f(\x_{k-1},\y_k)+\gamma_k f(\tilde{\x},\y_k)+\gamma_k\nabla_\x f(\z_k,\y_k)^\top(\v_k-\tilde{\x})+\frac{\gamma_k^2L}{2}\|\v_k-\v_{k-1}\|^2. $$
		Let $\d_k = \nabla_\x f_{\cP_k}(\z_k,\y_k)$. Notice that the update of $\v_k$ and the stopping condition of CndG procedure ensures that
		
		\begin{equation*}
		\max_{\x\in\cX} \langle \d_k+\alpha_k(\v_k-\v_{k-1}), \v_k-\x \rangle\leq \zeta_k.
		\end{equation*} 
		Combining the previous two inequality, we can obtain
		\begin{equation} \label{eq:10}
		\begin{split}
		f(\x_k,\y_k)\leq& (1-\gamma_k) f(\x_{k-1},\y_k)+\gamma_k f(\tilde{\x},\y_k)+\gamma_k\d_k^\top(\v_k-\tilde{\x})+\frac{\gamma_k^2L}{2}\|\v_k-\v_{k-1}\|^2\\
		&+\gamma_k(\nabla_\x f(\z_k,\y_k)-\d_k)^\top(\v_k-\tilde{\x})\\
		\leq& (1-\gamma_k) f(\x_{k-1},\y_k)+\gamma_k f(\tilde{\x},\y_k)+\gamma_k\zeta_k-\gamma_k \alpha_k(\v_k-\v_{k-1})^\top(\v_k-\tilde{\x})\\
		&+\frac{\gamma_k^2L}{2}\|\v_k-\v_{k-1}\|^2+\gamma_k(\nabla_\x f(\z_k,\y_k)-\d_k)^\top(\v_k-\tilde{\x})\\
		=& (1-\gamma_k) f(\x_{k-1},\y_k)+\gamma_k f(\tilde{\x},\y_k)+\gamma_k\zeta_k+\frac{\gamma_k \alpha_k}{2}(\|\v_{k-1}-\tilde{\x}\|^2-\|\v_k-\tilde{\x}\|^2)\\
		&+\frac{\gamma_k}{2}\left((L\gamma_k- \alpha_k)\|\v_k-\v_{k-1}\|^2+2(\d_k-\nabla_\x f(\z_k,\y_k))^\top(\v_{k-1}-\v_k)\right)\\
		&+\gamma_k(\d_k-\nabla_\x f(\z_k,\y_k))^\top(\tilde{\x}-\v_{k-1})\\
		\end{split}
		\end{equation}
		
		Due to the fact $ \alpha_k\geq L\gamma_k$ and $-\|\a\|^2+2\a^\top\b\leq \|\b\|^2$, we have
		\begin{equation} \label{eq:11}
		(L\gamma_k- \alpha_k)\|\v_k-\v_{k-1}\|^2+2(\d_k-\nabla_\x f(\z_k,\y_k))^\top(\v_{k-1}-\v_k)\leq \frac{\|\d_k-\nabla_\x f(\z_k,\y_k)\|^2}{ \alpha_k-L\gamma_k}
		\end{equation}
		
		Combining Eq.(\ref{eq:10}) and Eq.(\ref{eq:11}), we get
		
		\begin{equation*}
		\begin{split}
		\BE[f(\x_k,\y_k)-f(\tilde{\x},\y_k)]\leq& (1-\gamma_k) \BE[f(\x_{k-1},\y_k)-f(\tilde{\x},\y_k)]+\gamma_k\zeta_k\\
		&+\frac{\gamma_k \alpha_k}{2}\BE[\|\v_{k-1}-\tilde{\x}\|^2-\|\v_k-\tilde{\x}\|^2]+\frac{\gamma_k\BE[\|\d_k-\nabla_\x f(\z_k,\y_k)\|^2]}{ 2(\alpha_k-L\gamma_k)}   
		\end{split}
		\end{equation*}
		
		Let $\delta_k = \d_k-\nabla_\x f(\z_k,\y_k)$, then $\BE[\|\delta_k\|^2]\leq\sigma_k^2$. Let $\Phi(k)=k(k+1)(k+2)\BE[f(\x_k,\y_k)-f(\tilde{\x},\y_k)]$. According to Lemma \ref{lem:sprox}, we have $\BE[f(\x_k,\y_k)-\max_{\y\in\cY}f(\x_k,\y)]\geq -\epsilon_k$.
		Thus, we can obtain
		{\small\begin{equation} \label{eq:19}
			\begin{split}
			\Phi(k)\leq& \Phi(k-1)+k(k-1)(k+1)\BE[f(\x_{k-1},\y_k)-f(\x_{k-1},\y_{k-1})-f(\tilde{\x},\y_k)+f(\tilde{\x},\y_{k-1})]\\
			&+k(k+1)(k+2)\gamma_k\left(\zeta_k+\frac{ \alpha_k}{2}\BE[\|\v_{k-1}-\tilde{\x}\|^2-\|\v_k-\tilde{\x}\|^2]+\frac{\sigma_k^2}{2( \alpha_k-L\gamma_k)}\right)\\
			\leq& \Phi(k-1)+k(k-1)(k+1)\BE[\epsilon_{k-1}-f(\tilde{\x},\y_k)+f(\tilde{\x},\y_{k-1})]\\
			&+k(k+1)(k+2)\gamma_k\left(\zeta_k+\frac{ \alpha_k}{2}\BE[\|\v_{k-1}-\tilde{\x}\|^2-\|\v_k-\tilde{\x}\|^2]+\frac{\sigma_k^2}{ 2(\alpha_k-L\gamma_k)}\right)\\
			\leq& \Phi(0)+\sum_{s=1}^k s(s-1)(s+1)\epsilon_{s-1}-\sum_{s=1}^k s(s-1)(s+1)\BE[f(\tilde{\x},\y_s)-f(\tilde{\x},\y_{s-1})]\\
			&+ \sum_{s=1}^k s(s+1)(s+2)\gamma_s\left(\zeta_s+\frac{ \alpha_s}{2}\BE[\|\v_{s-1}-\tilde{\x}\|^2-\|\v_s-\tilde{\x}\|^2]+\frac{\sigma_s^2}{ 2(\alpha_s-L\gamma_s)}\right)\\
			=& \sum_{s=1}^k s(s-1)(s+1)\epsilon_{s-1}+3\sum_{s=1}^{k-1} s(s+1)\BE[f(\tilde{\x},\y_s)]-k(k-1)(k+1)\BE[f(\tilde{\x},\y_k)]\\
			&+ \sum_{s=1}^k s(s+1)(s+2)\gamma_s\left(\zeta_s+\frac{ \alpha_s}{2}\BE[\|\v_{s-1}-\tilde{\x}\|^2-\|\v_s-\tilde{\x}\|^2]+\frac{\sigma_s^2}{ 2(\alpha_s-L\gamma_s)}\right)\\
			\end{split}
			\end{equation}}
		
		According to the parameter setting, we have 
		\begin{align} \label{eq:20}
		\sum_{s=1}^k s(s+1)(s+2)\gamma_s\zeta_{s}=\frac{1}{192}k LD_\cX^2, \quad \sum_{s=1}^k s(s+1)(s+2)\frac{\gamma_s\sigma_s^2}{ 2(\alpha_s-L\gamma_s)}\leq k LD_\cX^2
		\end{align}
		We also have
		\begin{equation} \label{eq:21}
		\begin{split}
		&\sum_{s=1}^k s(s+1)(s+2)\frac{\gamma_s \alpha_s}{2}(\|\v_{s-1}-\tilde{\x}\|^2-\|\v_s-\tilde{\x}\|^2)\\
		=& 9\kappa L\sum_{s=1}^k s(\|\v_{s-1}-\tilde{\x}\|^2-\|\v_s-\tilde{\x}\|^2)\\
		\leq& 9\kappa L\sum_{s=0}^{k-1} \|\v_s-\tilde{\x}\|^2\leq 9k\kappa LD_\cX^2.
		\end{split}
		\end{equation}
		Substituting  Eq.(\ref{eq:20}) and Eq.(\ref{eq:21}) into Eq.(\ref{eq:19}) and using the fact that $\Phi(0)=0$, we have
		$$ \Phi(k)\leq \sum_{s=1}^k s(s-1)(s+1)\epsilon_{s-1}+3\sum_{s=1}^{k-1} s(s+1)\BE[f(\tilde{\x},\y_s)]-k(k-1)(k+1)\BE[f(\tilde{\x},\y_k)]+ 11k\kappa LD_\cX^2.$$
		
		Then, for any $\tilde{\y}\in\cY$, we have 
		\begin{equation*}
		\begin{split}
		&\sum_{s=1}^k s(s-1)(s+1)\epsilon_{s-1} + 11k\kappa LD_\cX^2\\
		\geq& \Phi(k)-3\sum_{s=1}^{k-1} s(s+1)\BE[f(\tilde{\x},\y_s)]+k(k-1)(k+1)\BE[f(\tilde{\x},\y_k)]\\\
		=& k(k+1)(k+2)\BE[f(\x_k,\y_k)]-3\sum_{s=1}^{k} s(s+1)\BE[f(\tilde{\x},\y_s)]\\\
		\geq& k(k+1)(k+2)\BE[f(\x_k,\y_k)-f(\tilde{\x},\bar{\y}_k)]\\
		\geq& k(k+1)(k+2)\BE[f(\x_k,\tilde{\y})-f(\tilde{\x},\bar{\y}_k)-\epsilon_k]
		\end{split}
		\end{equation*}
		where the second to last inequality is by the concavity of $f(\x,\cdot)$ and the definition of $\bar{\y}_k$.
		Then, we can obtain the bound of the expectation of the primal-dual gap:
		\begin{equation*}
		\BE[f(\x_k,\tilde{\y})-f(\tilde{\x},\bar{\y}_k)]\leq \frac{1}{k(k+1)(k+2)}\left(\sum_{s=1}^k s(s+1)(s+2)\epsilon_s+11k\kappa LD_\cX^2  \right)\leq \frac{12\kappa LD_\cX^2}{(k+1)(k+2)}
		\end{equation*}
		where the last equation comes from the fact that $\epsilon_k=\frac{\kappa LD_\cX^2}{k(k+1)(k+2)}$.
		
		%If $\epsilon_k=\frac{\kappa LD_\cX^2}{k(k+1)(k+2)}$, we have $f(\x_k,\y^*)-f(\tilde{\x},\bar{\y}_k)\leq \frac{16\kappa LD_\cX^2}{(k+1)(k+2)}$
		
	\end{proof}
	
	\subsection{Proof of Corollary \ref{col:mpscgs-f}}
	If the objective function is of finite-sum form, the performance of iSTORC is the same as STORC. According to Corollary 2 of ~\cite{hazan2016variance}, the iSTORC algorithm requires  $\cO\left((n+\kappa^2)\log\frac{ \kappa LD_\cY^2}{\epsilon_{k}}\right)$ IFO calls and $\cO\left(\frac{\kappa LD_\cY^2}{\epsilon_{k}}\right)$ LO calls. Thus, the IFO and LO complexity of MPCGS are respectively 
	$\cO\left(\sum_{k=1}^N(P_k+\sum_{r=1}^R (n+
	\kappa^2)\log\frac{\kappa LD_\cY^2}{\epsilon_k})\right) $ and
	$\cO\left(\sum_{k=1}^N\sum_{r=1}^R \left(\frac{\kappa LD_\cY^2}{\epsilon_k}+\frac{\beta_k D_\cX^2}{\zeta_k}\right)\right)$, where $R$ is the number of iterations of Stochastic-Prox-step.
	Theorem \ref{thm:mpscgs} implies that $N$ should be the order $\Theta\left(\sqrt{\frac{\kappa LD_\cX^2}{\epsilon}}\right)$. Plugging in all parameters obtains the complexity of MPSCGS.

	\subsection{Proof of Corollary \ref{col:mpscgs-e}}
	Suppose the objective function has expectation form. According to Corollary \ref{col:istorc}, the the iSTORC algorithm requires  $\cO\left((
	\sqrt{\kappa}/(L\epsilon_k)+\kappa^2)\log\frac{ \kappa LD_\cY^2}{\epsilon_{k}}\right)$ SFO calls and $\cO\left(\frac{\kappa LD_\cY^2}{\epsilon_{k}}\right)$ LO calls. The rest analysis follows the proof of Corollary \ref{col:mpscgs-f}.

\end{document}

%% file: math.tex
\usepackage{amsmath}\allowdisplaybreaks
\usepackage{amsfonts,bm}
\usepackage{amssymb}
\usepackage{amsthm}
\usepackage{algorithm}
\usepackage{algorithmic}

\theoremstyle{plain}
\newtheorem{theorem}{Theorem}%[section]
\newtheorem{definition}{Definition}

\newtheorem{lemma}{Lemma}
\newtheorem{assumption}{Assumption}
\newtheorem{remark}{Remark}
\newtheorem{corollary}{Corollary}

\def\vnu{{\bm{\nu}}}

\def\a{{\bf{a}}}
\def\b{{\bf{b}}}

\def\d{{\bf{d}}}

\def\g{{\bf{g}}}

\def\p{{\bf{p}}}
\def\q{{\bf{q}}}
\def\r{{\bf{r}}}

\def\u{{\bf{u}}}
\def\v{{\bf{v}}}
\def\w{{\bf{w}}}
\def\x{{\bf{x}}}
\def\y{{\bf{y}}}
\def\z{{\bf{z}}}

\def\cD{{\mathcal{D}}}

\def\cG{{\mathcal{G}}}

\def\cO{{\mathcal{O}}}
\def\cP{{\mathcal{P}}}
\def\cQ{{\mathcal{Q}}}

\def\cS{{\mathcal{S}}}

\def\cX{{\mathcal{X}}}
\def\cY{{\mathcal{Y}}}

\def\BE{{\mathbb{E}}}

\def\BR{{\mathbb{R}}}

\def\X {{\bf X}}

\DeclareMathOperator*{\argmax}{arg\,max}
\DeclareMathOperator*{\argmin}{arg\,min}